\documentclass[a4paper,oneside,10pt,oldfontcommands]{article}
\usepackage[a4paper, total={426pt, 674pt}]{geometry}

\newcommand{\langue}{anglais}	

\usepackage{ifthen}
\ifthenelse{\equal{\langue}{francais}}{
	\usepackage[french]{babel}}
	{\ifthenelse{\equal{\langue}{anglais}}{\usepackage[english]{babel}}{}}
\usepackage[T1]{fontenc}
\usepackage[utf8]{inputenc}
\usepackage[babel,german=guillemets]{csquotes}
\usepackage{eso-pic}

\usepackage{lmodern} 
\usepackage{enumitem}
\usepackage{ifthen}
\ifthenelse{\equal{\langue}{francais}}{
	\frenchbsetup{StandardLists=true}}
	{}

\usepackage{graphicx} 
\usepackage{subcaption}

\usepackage{amsmath}
\usepackage{amssymb}
\usepackage{amsthm}
\usepackage{amsfonts}
\usepackage{bbold}
\usepackage{mathtools}
\usepackage{stmaryrd}
\usepackage{tikz-cd}
\usepackage{hyperref}
\usepackage{multirow}

\ifthenelse{\equal{\langue}{francais}}{
	\newcommand{\theoremenom}{Théorème}
	\newcommand{\propositionnom}{Proposition}
	\newcommand{\lemmenom}{Lemme}
	\newcommand{\corollairenom}{Corollaire}
	\newcommand{\definitionnom}{Définition}
	\newcommand{\remarquenom}{Remarque}
	\newcommand{\exemplenom}{Exemple}
	\newcommand{\conjecturenom}{Conjecture}
}{\ifthenelse{\equal{\langue}{anglais}}{
	\newcommand{\theoremenom}{Theorem}
	\newcommand{\propositionnom}{Proposition}
	\newcommand{\lemmenom}{Lemma}
	\newcommand{\corollairenom}{Corollary}
	\newcommand{\definitionnom}{Definition}
	\newcommand{\remarquenom}{Remark}
	\newcommand{\exemplenom}{Example}
	\newcommand{\conjecturenom}{Conjecture}
}{}}
	
\newtheorem{theoreme}{\theoremenom}
\newtheorem{proposition}[theoreme]{\propositionnom}
\newtheorem{lemme}[theoreme]{\lemmenom}
\newtheorem{corollaire}[theoreme]{\corollairenom}
\newtheorem{definition}[theoreme]{\definitionnom}
\newtheorem{remarque}[theoreme]{\remarquenom}

\newtheorem*{conjecture}{\conjecturenom}

\usepackage{thmtools}
\usepackage{thm-restate}

\makeatletter
\def\cleartheorem#1{%
    \expandafter\let\csname#1\endcsname\relax
    \expandafter\let\csname c@#1\endcsname\relax
}
\makeatother
\newcommand{\compteurThm}{1}
\newcounter{annexe}

\newcommand{\R}{\mathbb{R}}
\newcommand{\N}{\mathbb{N}}

\newcommand{\1}{\mathbb{1}}

\begin{document}

\pagestyle{empty} 


\title{Towards the optimality of the ball for the Rayleigh Conjecture concerning the clamped plate}
\author{
Roméo Leylekian
\footnote{Aix-Marseille Université, CNRS, I2M, Marseille, France - \textbf{email:} romeo.leylekian@univ-amu.fr}
}
\date{} 
\maketitle


\begin{abstract}
In 1995, Nadirashvili and subsequently Ashbaugh and Benguria proved the Rayleigh Conjecture concerning the first eigenvalue of the bilaplacian with clamped boundary conditions in dimension $2$ and $3$. Since then, the conjecture has remained open in dimension $d>3$. In this document, we contribute in answering the conjecture under a particular assumption regarding the critical values of the optimal eigenfunction. More precisely, we prove that if the optimal eigenfunction has no critical value except its minimum and maximum, then the conjecture holds. This is performed thanks to an improvement of Talenti's comparison principle, made possible after a fine study of the geometry of the eigenfunction's nodal domains. 
\end{abstract}




\pagestyle{plain} 


\section{Introduction}

The purpose of this paper is to tackle the Rayleigh Conjecture for the bilaplacian with Dirichlet boundary conditions (also called Dirichlet bilaplacian). More precisely, we consider the following eigenvalue problem defined in an arbitrary bounded open set $\Omega\subseteq\R^d$, $d\in\N^*$.
\begin{equation}\label{eq:equation aux vp}
\left\{
\begin{array}{rcll}
\Delta^2 u & = & \Gamma u & in\quad \Omega, \\
u & = & 0 & on\quad \partial\Omega,\\
\partial_n u & = & 0 & on\quad\partial\Omega,
\end{array}
\right.
\end{equation}
where $\partial_n=\vec{n}\cdot\nabla$ denotes the partial derivative in the direction of the outward normal unit vector $\vec{n}$. Then, $(u,\Gamma)$ is a solution to (\ref{eq:equation aux vp}) whenever $\Gamma\in\R$ and $u\in H_0^2(\Omega)$ satisfy the first equation in the sense of distributions. General results on polyharmonic operators, and especially on the bilaplacian, can be found in \cite{gazzola-grunau-sweers}. In particular, it is known that the Dirichlet bilaplacian turns out to have a compact self-adjoint positive resolvent when read as an operator on $L^2(\Omega)$. Consequently, there exists an unbounded sequence of positive eigenvalues $\Gamma$ satisfying (\ref{eq:equation aux vp}). Here, we are interested in the lowest one, that we denote $\Gamma(\Omega)$ since it depends on the domain $\Omega$. Due to the min-max formulation of eigenvalues provided by spectral theory, it is known that $\Gamma(\Omega)$ shall be characterised variationally in the following way:

\begin{equation}\label{eq:formulation variationnelle}
\Gamma(\Omega)=\min_{\substack{u\in H_0^2(\Omega)\\ u\neq 0}}\frac{\int_\Omega(\Delta u)^2}{\int_\Omega u^2}.
\end{equation}

The Rayleigh Conjecture concerns the problem of determining the open set with least first eigenvalue among all open sets having same measure. As its counterpart regarding the Dirichlet Laplacian, the conjecture states that the optimal set is given by any ball complying with the volume constraint. Moreover, uniqueness is also claimed in a sense that we shall precise. Indeed, recall that, besides translations, $\Gamma(\Omega)$ is invariant up to the removal of a set of zero $H^2$-capacity (see \cite[section 3.8.1]{henrot-pierre}). In other words, if $|.|$ stands for the Lebesgue measure,

\begin{conjecture}
Let\/ $\Omega$ be a bounded open subset of\/ $\R^d$ and $B$ a ball such that $|B|=|\Omega|$. Then,
\begin{equation}\label{eq:conjecture}
\Gamma(\Omega)\geq\Gamma(B).
\end{equation}
Moreover there is equality if and only if\/ $\Omega$ is a ball, up to a set of zero $H^2$-capacity.
\end{conjecture}

Let us briefly recall the state of the art on the topic and refer to \cite{gazzola-grunau-sweers} for further literature. After its publication in 1894, one of the first results regarding the conjecture goes back to Szegö \cite{szego}, who proved that, as long as the first eigenfunction is of fixed sign, the Faber-Krahn type inequality (\ref{eq:conjecture}) holds. Unfortunately, - and it is one of the main challenges when dealing with higher order elliptic operators - it turns out that the bilaplacian does not enjoy the maximum principle in general for arbitrary domains. In particular, the one-sign property of the first eigenfunction is not guaranteed, and indeed fails for some domains. This behaviour was observed for the first time in annuli with small inner radius in 1952 \cite{duffin-shaffer,coffman-duffin-shaffer}. However, writing the optimality conditions of first and second order, Mohr \cite{mohr} showed in 1975 that no other planar regular shape than the disk could be optimal. Unfortunately, it seems that the approach of Mohr heavily relies on the fact that the ambient space is $\R^2$. Moreover, one year later, Talenti proved a far-reaching comparison principle in \cite{talenti76} from which arose a series of papers making Mohr's result fall into disuse. Indeed, a refinement of its comparison principle led Talenti to exhibit from the variational formulation (\ref{eq:formulation variationnelle}) an auxiliary problem, today called the \enquote{two-ball problem}. This allowed him to give a lower bound on the optimal eigenvalue depending on the dimension (see \cite{talenti81}). Then, in 1995, Nadirashvili \cite{nadirashvili} solved the two-ball problem and hence the Rayleigh Conjecture in $\R^2$. Subsequently, still in the wake of Talenti's approach, Ashbaugh and Benguria proved the conjecture in $\R^2$ and $\R^3$ in 1995 (see \cite{ashbaugh-benguria}). Moreover, the two-ball problem was finally solved in any dimension in 1996 by Ashbaugh and Laugesen in \cite{ashbaugh-laugesen}. As a consequence, they provided a very precise lower bound on the optimal eigenvalue. As a payback, however, they showed that the plain approach of Talenti could not answer the Rayleigh Conjecture when $d>3$. Since then, very few progress has been performed. Let us quote however the recent works \cite{kristaly20,kristaly22} proving an analogue of the Rayleigh Conjecture in negatively and positively curved Riemanian manifolds. Finally, we mention our previous work \cite{leylekian} in which we provide sufficient conditions for the conjecture to hold.

The present document is intended to contribute for the resolution of the Rayleigh Conjecture. More precisely, we will prove that the conjecture holds under some mere assumptions that we shall describe below. The first of our assumptions is that there exists a $C^4$ regular shape solving
\begin{equation}\label{eq:pb}
\min\{\Gamma(\Omega):\Omega\subseteq \R^d\text{ bounded open set, }|\Omega|=c\}
\end{equation}
where $c$ is a given positive real number. We precise that the existence of an optimal shape for (\ref{eq:pb}) is still an open issue. See however \cite[Theorem 3.5]{bucur} and more recently \cite{stollenwerk} dealing with the question for shapes lying within some prescribed bounded region. See also \cite[section 2.4]{bucur-buttazzo}, in which shapes are sought in a particular class of convex domains. In the rest of the document, we will denote $\Omega$ a $C^4$ regular solution to (\ref{eq:pb}). The assumption concerning the regularity of $\Omega$ is standard when one invokes shape derivation as it will be done. Indeed, the fact that $\Omega$ is $C^4$ guarantees that the eigenfunctions are $H^4(\Omega)$ (see \cite[Theorem 2.20]{gazzola-grunau-sweers}). Furthermore, the $L^p$ regularity theory shows that the eigenfunctions are even $W^{4,p}(\Omega)$ for all $1<p<\infty$. This, combined with Sobolev embeddings, gives that the eigenfunctions are $C^{3,\alpha}(\overline{\Omega})$, for all $0<\alpha<1$. This observation will be useful at some point to study the behaviour of the first eigenfunction near the boundary of $\Omega$ (see Proposition \ref{prop:positivité au voisinage du bord}).

Apart from the regularity assumption on $\Omega$, we will also need a non standard technical assumption. This condition regards the critical points of a first eigenfunction on $\Omega$ (we will denote it $u$ in the rest of the document), and is intended to ensure that the topology of the upper level sets of $u$ does not change (see also Remark \ref{remarque: U}). It can be expressed as follows.
\begin{equation}\label{hyp:u}
\tag{U}
\forall x\in\Omega,\qquad\nabla u(x)=0\quad\Longrightarrow\quad u(x)\textit{ is the global minimum or global maximum of u}.
\end{equation}
Let us emphasize that, although being critical for $u$, the points of $\partial\Omega$ are not taken into consideration in (\ref{hyp:u}). In other words, (\ref{hyp:u}) is equivalent to ask that $u$ admits only $\min\limits_{\overline{\Omega}} u$ and $\max\limits_{\overline{\Omega}} u$ as possible critical values within $\Omega$. Note however that it allows $u$ to have many critical points. This occurs for instance in annuli with large inner radius, for which the first eigenfunction is radially symmetric and admits only one critical value in $\Omega$, which is its maximum (or minimum) value (see \cite{coffman-duffin-shaffer}). In this situation, we see that the corresponding critical points form a sphere due to the radial symmetry of $u$. Using assumption (\ref{hyp:u}), it becomes possible to express the main conclusion of the present document, which is the theorem stated below.

\newcommand{\reg}{$C^4$}
\begin{theoreme}\label{thm:faber-krahn U}
Let\/ $\Omega$ be a \reg\/ optimal shape for problem (\ref{eq:pb}) and $B$ a ball such that $|\Omega|=|B|$. Assume that an eigenfunction associated with $\Gamma(\Omega)$ satisfies (\ref{hyp:u}) and that $4\leq d \leq 9$. Then, \begin{equation*}
\Gamma(\Omega)\geq\Gamma(B).
\end{equation*}
Moreover there is equality if and only if\/ $\Omega$ is a ball.
\end{theoreme}

\begin{remarque}
The assumption on the dimension $d$ comes from the fact that Theorem \ref{thm:faber-krahn U} relies on Theorem \ref{thm:pb aux 2 boules}, the assumptions of which are technical inequalities that one may verify numerically. They are not restrictive, and have been checked for dimensions $4\leq d \leq 9$ (see section \ref{sec:numerique} in appendix). In other words, in Theorem \ref{thm:faber-krahn U}, the assumption $4\leq d \leq 9$ shall be replaced with the assumptions of Theorem \ref{thm:pb aux 2 boules}.
\end{remarque}

The proof of Theorem \ref{thm:faber-krahn U} is based on a refinement of the classical proof given by Ashbaugh and Benguria \cite{ashbaugh-benguria}, following the pioneer approach of Talenti \cite{talenti81}. Let us briefly explain the structure of our argumentation. The shape derivation framework allows to derive optimality conditions for the optimal shape. These conditions read in terms of the Laplacian of the eigenfunction on the boundary. In section \ref{sec:simplicity}, we exploit this particular information to conclude that the eigenfunction does not change sign in a neighbourhood of the boundary. One of our contributions is to translate this observation in terms of the geometry of the nodal domains of the eigenfunction. This is performed in section \ref{sec:geometrie à trou}, where we show that one of the two nodal domains has a \enquote{hole} containing the other one. The purpose is then to inject the special geometry of the nodal domains in Talenti's comparison principle in order to derive a refined version of this principle. This is obtained in section \ref{sec:talenti}. Our improved comparison principle is used instead of the standard one in the classical proof and yields an asymmetric two-ball problem as shown in section \ref{sec:probleme aux deux boules}. The resolution of the new two-ball problem is carried out in section \ref{sec:resolution deux boules}. The proof of Theorem \ref{thm:faber-krahn U} is finally achieved in section \ref{sec:preuve theoreme}, in which we also discuss our main assumption (\ref{hyp:u}).

\begin{remarque}
Let us point out that the only part of the proof of Theorem \ref{thm:faber-krahn U} which crucially relies on assumption (\ref{hyp:u}) is the refined Talenti type comparison principle. Indeed, assumption (\ref{hyp:u}) is made on the eigenfunction $u$ in order to ensure that the special geometry of the nodal domains of $u$ is transmitted to the upper level sets of $u$ (see Remark \ref{remarque: U}), allowing to obtain a quantitative version of Talenti's comparison principle (see Theorem \ref{thm:talenti modifie}). However, it is clear that, even without assumption (\ref{hyp:u}), the standard version of Talenti's comparison principle is not optimal when applied to an open set with a hole. Therefore, it would be of particular interest to study to which extent assumption (\ref{hyp:u}) can be removed in Theorem \ref{thm:talenti modifie}. This could in turn allow to remove assumption (\ref{hyp:u}) in Theorem \ref{thm:faber-krahn U}.
\end{remarque}

\section{Simplicity of the eigenvalue and sign of an eigenfunction near the boundary}\label{sec:simplicity}

In this section we recall the optimality condition satisfied by an optimal shape for problem (\ref{eq:pb}). From this, we deduce an immediate but not less crucial consequence reading in terms of the sign of the first eigenfunction near any connected component of the boundary. Let us begin with the following result (see \cite[Theorem 3]{leylekian}) which comes from the theory of shape derivation.

\begin{theoreme}\label{thm:condition d'optimalité}
Let\/ $\Omega$ be a \reg\/ open set solving (\ref{eq:pb}). Then, $\Gamma(\Omega)$ is simple. Moreover, if $u$ denotes an $L^2$-normalised eigenfunction associated with $\Gamma(\Omega)$, $\Delta u$ is a.e. constant equal to $\pm\alpha$ on any connected component of $\partial\Omega$, where
\begin{equation}\label{eq:alpha}
\alpha:=\sqrt{\frac{4\Gamma(\Omega)}{d|\Omega|}}.
\end{equation}
\end{theoreme}
Note that the optimality condition given in Theorem \ref{thm:condition d'optimalité} is actually fulfilled by any \reg\/ regular shape $\Omega$ being a critical shape of $\Omega\mapsto |\Omega|^{\frac{4}{d}}\Gamma(\Omega)$ in the sense of shape derivatives. This motivates the following definition.

\begin{definition}\label{def:criticalité}
A bounded open set\/ $\Omega$ is a critical shape (for the first eigenvalue) if any $L^2$-normalised first eigenfunction $u$ on $\Omega$ is such that $\Delta u$ is a.e. constant equal to $\pm\alpha$ on each connected component of $\partial\Omega$, where $\alpha$ is given in (\ref{eq:alpha}).
\end{definition}
We remark that it is easy to show that any ball is a critical shape (see \cite{buoso-lamberti13} for general results on this topic). Now, let us state an easy but important lemma regarding the sign of the eigenfunction near boundary components of a critical shape. We mention that this result seems quite known in the litterature (see \cite[equation (20)]{mohr} and also \cite[equation (4.26)]{eichmann-schatzle}).

\begin{lemme}\label{lemme:positivité au voisinage du bord}
Let\/ $\Omega$ be a $C^1$ critical shape and $\gamma$ a connected component of $\partial\Omega$. Assume that a first eigenfunction $u$ is $C^2(\overline{\Omega})$. Then there exists a neighbourhood of $\gamma$ in which the sign of $u$ is the same as the sign of $\Delta u$ on $\gamma$.
\end{lemme}

\begin{proof}
Considering $-u$ we can assume without loss of generality that $\Delta u=\alpha>0$ on $\gamma$. Thanks to the ancillary information that $u=|\nabla u|=0$ on the boundary, we get that $u>0$ in a neighbourhood of $\gamma$. This comes from the fact that for each point $p\in\gamma$, the regularity of $u$ allows writing the following Taylor expansion at $p$: for all $x$ close enough to $p$ such that $(x-p)\parallel\vec{n}(p)$,
\begin{equation}\label{eq:taylor}
u(x)\geq u(p)+\nabla u(p)\cdot(x-p)+\frac{1}{2}D^2 u(p)\cdot (x-p)^2-\frac{\alpha}{4}|x-p|^2=\frac{\alpha}{4}|x-p|^2>0.
\end{equation}
For the moment, the closeness of $x$ and $p$ depends on $p$. But as $\partial\Omega$ is compact, the associated threshold might be chosen uniformly, meaning that there exists $\epsilon>0$ such that for all $p\in\partial\Omega$ and $x\in\Omega$, if $x\in B_\epsilon(p)$ and $(x-p)\parallel\vec{n}(p)$, then (\ref{eq:taylor}) holds. Then, $\{p-t\vec{n}(p): p\in\gamma, t\in]0,\epsilon[\}$ is a neighbourhood of $\gamma$ in which $u$ is positive.
\end{proof}

The main result of this section is a consequence of this lemma and of the optimality condition satisfied by $\Omega$.

\begin{proposition}\label{prop:positivité au voisinage du bord}
Let\/ $\Omega$ be a \reg\/ optimal shape for problem (\ref{eq:pb}). Then any first eigenfunction is of constant sign in a neighbourhood of any connected component of $\partial\Omega$.
\end{proposition}

\begin{proof}
Theorem \ref{thm:condition d'optimalité} shows that $\Omega$ is a critical shape. Let us justify that any first eigenfunction $u$ is $C^2(\overline{\Omega})$.  This comes from elliptic regularity and from a standard induction argument.

Indeed, we know that $u\in L^2(\Omega)$. Let us assume that $u\in L^p(\Omega)$ for some $p\geq2$. Because $\Omega$ is $C^4$, elliptic regularity (\cite[Theorem 2.20]{gazzola-grunau-sweers}) yields $u\in W^{4,p}(\Omega)$. Then, by Sobolev embeddings (\cite[Theorem 2.20]{gazzola-grunau-sweers}) $u\in L^{q}(\Omega)$ for all $1<q<p^*$, where $p^*=+\infty$ if $d/p\leq 4$, and $p^*=dp/(d-4p)$ otherwise. This suggests to define the sequence $(p_n)_n$ by $p_0:=2$, and
$$
p_{n+1}:=p_n^*=
\begin{cases}
+\infty & \text{if }4p_n\geq d, \\
dp_n/(d-4p_n) & \text{if }4p_n< d.
\end{cases}
$$
By induction, the previous discussion shows that for all $n\in\N$ and all $1<q<p_n$, $u\in L^q(\Omega)$. But the sequence $(p_n)_n$ is nondecreasing, hence converges in $[2,+\infty]$. If the limit $p$ were finite, it would satisfy $p=dp/(d-4p)$, which admits no solution in $[2,+\infty[$, hence $p=+\infty$.

In particular, $u\in L^p(\Omega)$ for all $1<p<\infty$, thus $u\in W^{4,p}(\Omega)$. Since, for all $0<\beta<1$, there exists $1<p<\infty$ such that $4-\frac{d}{p}\geq 3+\beta$, we obtain that $u\in C^{3,\beta}(\overline{\Omega})$, which shows as claimed that $u\in C^2(\overline{\Omega})$. Eventually, we apply Lemma \ref{lemme:positivité au voisinage du bord} and get the result.
\end{proof}

\section{Geometry of $\Omega_+$}\label{sec:geometrie à trou}

The proof of Theorem \ref{thm:faber-krahn U} is based on a refinement of Talenti's comparison principle which is used in the classical proof given in \cite{ashbaugh-benguria}. This refinement needs a crucial new information regarding the geometry of the nodal domains \mbox{$\Omega_+:=\{u>0\}$} and \mbox{$\Omega_-:=\{u<0\}$}. For that, one shall use advantageously Proposition \ref{prop:positivité au voisinage du bord}, which precisely tells that $\Omega_+$ and $\Omega_-$ cannot meet simultaneously a given connected component of the boundary. Therefore, as we will see in a few moment, if $\Omega$ is connected, and if $\Omega_+,\Omega_-\neq\emptyset$, at least one of $\Omega_+,\Omega_-$ will have a hole. In order to make this observation clearer, we need a precise definition of a hole.

\begin{definition}\label{def:trou}
Let $\omega$ and $T$ be nonempty bounded open sets. We say that $T$ is a hole for $\omega$ if:
\begin{enumerate}
\item $\omega$ and $T$ are disjoint,
\item $\partial T\subseteq \partial \omega$.
\end{enumerate}
\end{definition}

Let us remark that a hole need neither to be connected nor simply connected (see Figure \ref{fig:trous/trou non connexe}). On the other hand, unlike Figure \ref{fig:trous/trou connexe}, an open set admitting a hole shall be simply connected. This is the case with annuli in dimension $d>2$, which admit balls as holes and remain simply connected.

\begin{figure}[h!]
\centering
     \begin{subfigure}[b]{0.3\textwidth}
         \centering
         \includegraphics[width=\textwidth]{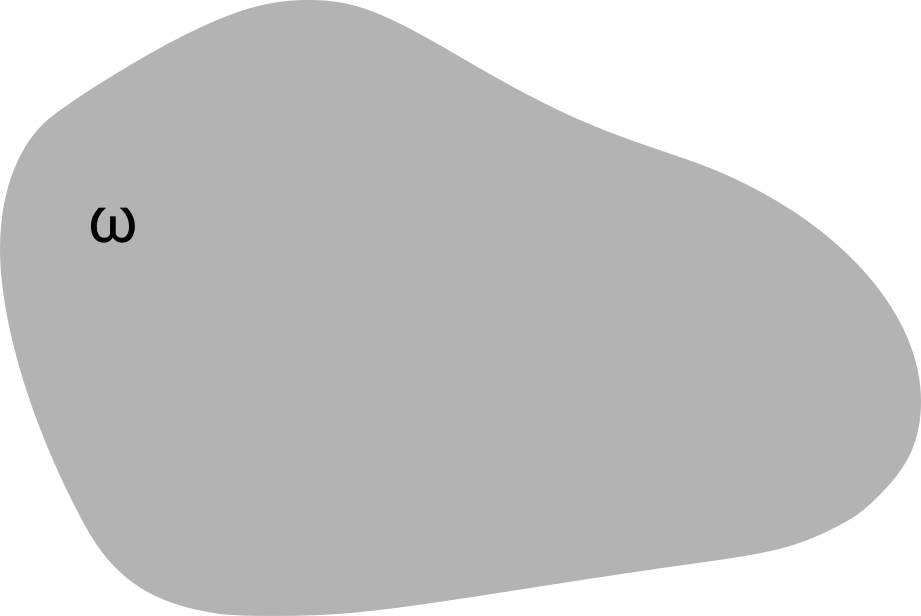}
         \caption{No hole.}
         \label{fig:trous/pas de trou}
     \end{subfigure}
     \hfill
     \begin{subfigure}[b]{0.3\textwidth}
         \centering
         \includegraphics[width=\textwidth]{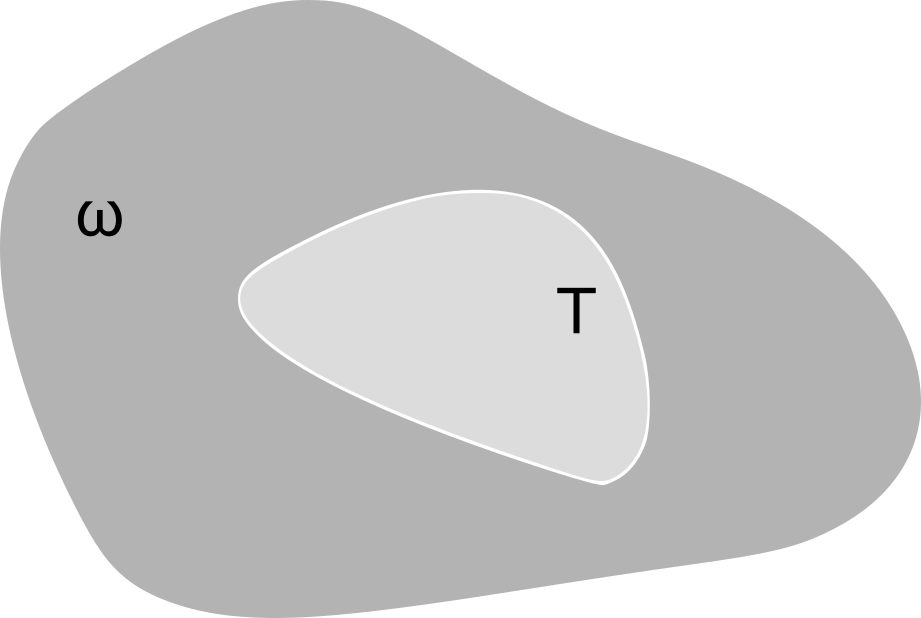}
         \caption{Connected hole.}
         \label{fig:trous/trou connexe}
     \end{subfigure}
     \hfill
     \begin{subfigure}[b]{0.3\textwidth}
         \centering
         \includegraphics[width=\textwidth]{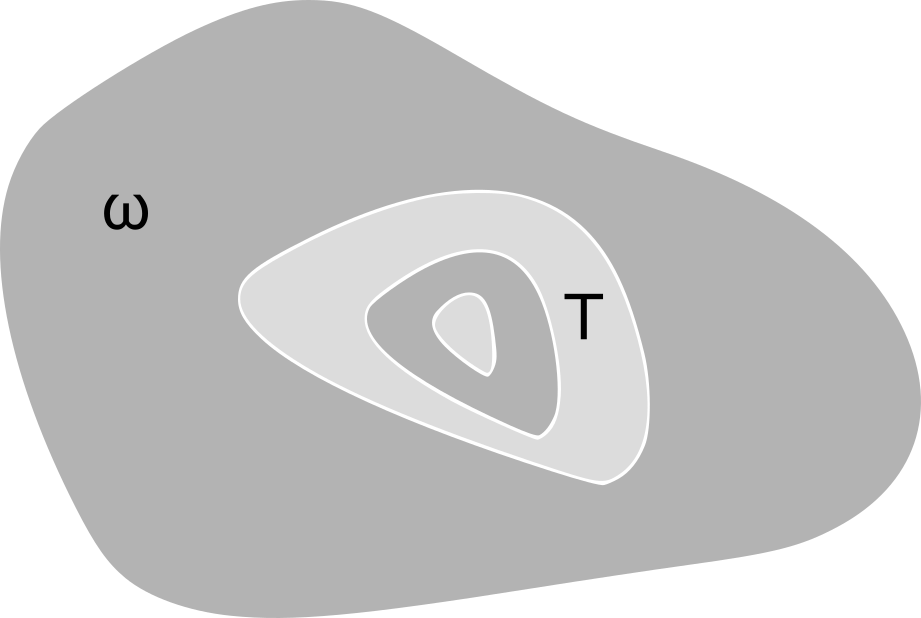}
         \caption{Disconnected hole.}
         \label{fig:trous/trou non connexe}
     \end{subfigure}
        \caption{Holes in different situations. In each case, the open set $\omega$ is in dark grey, whereas its hole $T$, if it exists, is in light gray.}
        \label{fig:trous1}
\end{figure}

Thanks to the next lemma, we see that it is always possible to assume that $\Omega_+$ has such a hole.

\begin{restatable}{relemme}{LemmeTrou}\label{lemme:trou}
Let $\omega_+,\omega_-$ be nonnempty disjoint open sets contained in a bounded connected open set $\omega$, such that, $\overline{\omega_+}\cup\overline{\omega_-}=\overline{\omega}$ and $\partial\omega\cap\partial\omega_+\cap\partial\omega_-=\emptyset$. We also assume that $\omega_\pm$ is $C^1$ relatively with respect to $\omega$, meaning that $\partial\omega_\pm\cap\omega$ is a $C^1$ hypersurface of $\R^d$. Then, either $\omega_+$ has a hole containing $\omega_-$ or $\omega_-$ has a hole containing $\omega_+$.
\end{restatable}

The technical proof is given in appendix, section \ref{annexe:trou}. Let us emphasize that the conclusion of Lemma \ref{lemme:trou} is that the hole of $\omega_+$ (resp. $\omega_-$) contains $\omega_-$ (resp. $\omega_+$). Indeed, in general, as it is shown in Figure \ref{fig:trous/omega_+ troué}, the hole is strictly greater than $\omega_-$ (resp. $\omega_+$).
Let us also briefly discuss the assumption of $C^1$ regularity of $\partial\omega_\pm\cap\omega$. We believe that this assumption shall be weakened, since it is only needed for applying Jordan-Brouwer Theorem and also for approaching $\partial\omega_\pm\cap\omega$ strictly from one side only (see Lemma \ref{lemme:connexe privé d'un ensemble ne recontrant pas sa frontiere} for details). We mention that this last point might be achieved when $\partial\omega_\pm\cap\omega$ is of finite perimeter \cite{schmidt}. However, it is not clear how weak the regularity of $\omega_\pm$ can be assumed to run the whole proof. Anyway, this is not an issue as, in the next lines, $\partial\omega_\pm\cap\omega$ will turn to be the nodal line of the first eigenfunction, and hence will enjoy some regularity properties. This is detailed in Corollary \ref{corollaire:trou}, which explains how Lemma \ref{lemme:trou} is applied to the nodal domains.

\begin{figure}[h!]
\centering
     \begin{subfigure}[b]{0.3\textwidth}
         \centering
         \includegraphics[width=\textwidth]{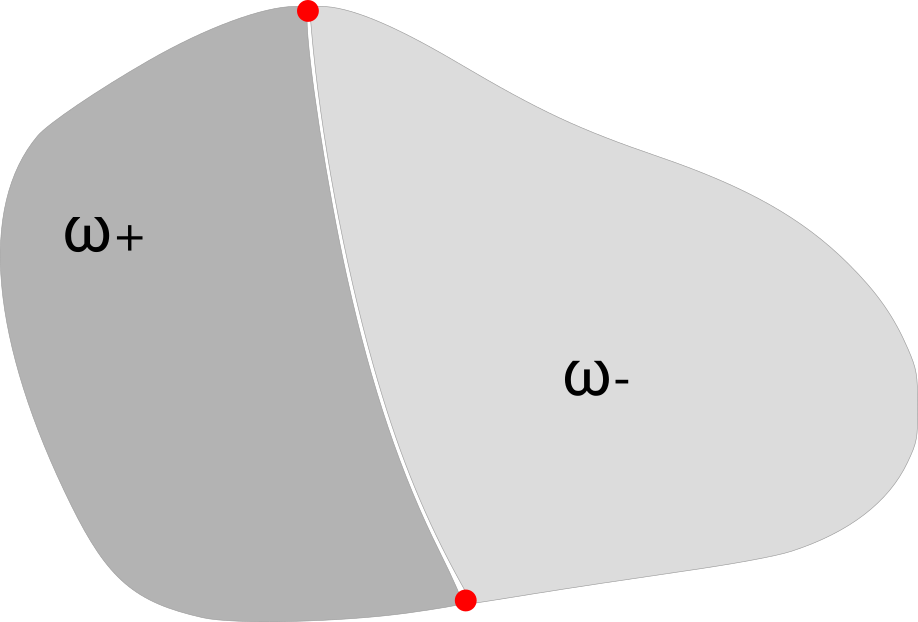}
         \caption{$\partial\omega\cap\partial\omega_+\cap\partial\omega_-\neq\emptyset$.}
         \label{fig:trous/intersection des bords}
     \end{subfigure}
     \hfill
     \begin{subfigure}[b]{0.3\textwidth}
         \centering
         \includegraphics[width=\textwidth]{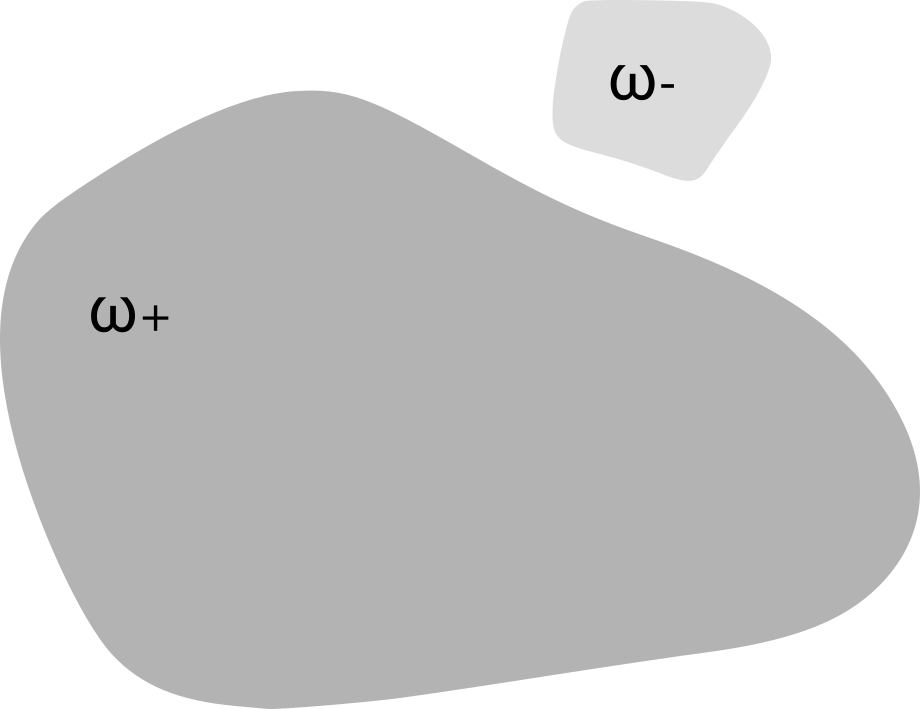}
         \caption{$\omega$ disconnected.}
         \label{fig:trous/omega non connexe}
     \end{subfigure}
     \hfill
     \begin{subfigure}[b]{0.3\textwidth}
         \centering
         \includegraphics[width=\textwidth]{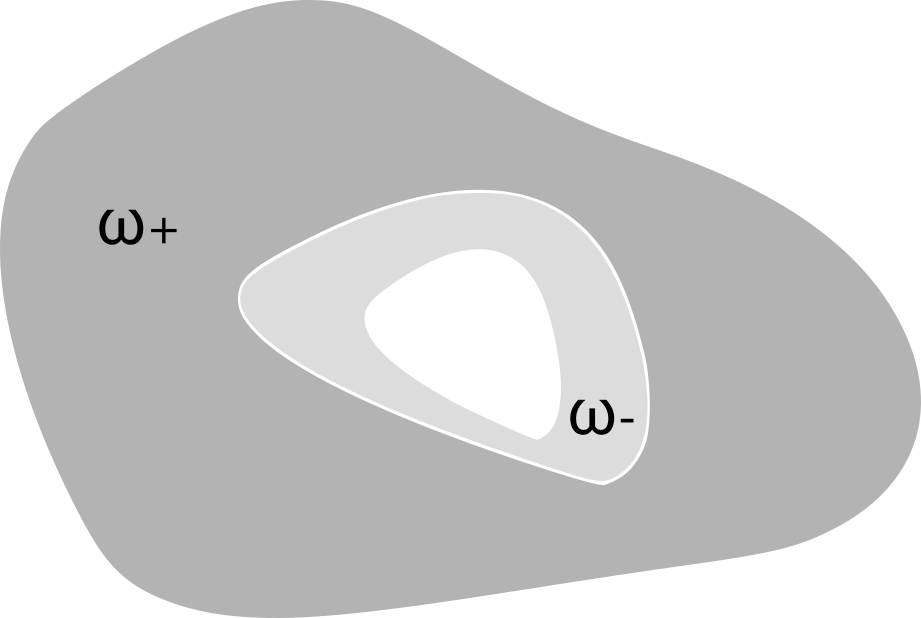}
         \caption{$\omega_+$ with a hole.}
         \label{fig:trous/omega_+ troué}
     \end{subfigure}
        \caption{Different configurations of the open sets $\omega_+$ and $\omega_-$. In each case, $\omega_+$ is in dark grey, and $\omega_-$ is in light gray. In situation \ref{fig:trous/intersection des bords}, the assumption $\partial\omega\cap\partial\omega_+\cap\partial\omega_-=\emptyset$ of Lemma \ref{lemme:trou} fails since $\partial\omega_+$ and $\partial\omega_-$ meet simultaneously $\partial\omega$ at the red points. In situation \ref{fig:trous/omega non connexe}, the assumption of connectedness of $\omega$ fails. In situation \ref{fig:trous/omega_+ troué}, all the assumptions of Lemma \ref{lemme:trou} hold, and hence $\omega_+$ admits a hole containing $\omega_-$.}
        \label{fig:trous2}
\end{figure}

\begin{corollaire}\label{corollaire:trou}
Let\/ $\Omega$ be a \reg\/ optimal shape for (\ref{eq:pb}), and $u$ a first eigenfunction on $\Omega$. Assume that $\Omega_+:=\{u>0\}$ and $\Omega_-:=\{u<0\}$ are both nonempty and that there is no critical point of $u$ of level zero in $\Omega$. Then, either $\Omega_+$ admits a hole containing $\Omega_-$ or $\Omega_-$ admits a hole containing $\Omega_+$.
\end{corollaire}

\begin{proof}
We mention that $\Omega$ is connected (otherwise there would exist a connected component $\omega$ of $\Omega$ such that $\Gamma(\omega)=\Gamma(\Omega)$, and because $|\omega|<|\Omega|$, the dilation of $\omega$ would have a lower first eigenvalue than $\Omega$, which would be a contradiction). Moreover, thanks to Proposition \ref{prop:positivité au voisinage du bord}, we have the key property $\partial\Omega\cap\partial\Omega_+\cap\partial\Omega_-=\emptyset$. The $C^1$ relative regularity of $\Omega_\pm$ with respect to $\Omega$ comes from the assumption regarding the critical points of $u$ and from the implicit function theorem. Therefore, all the hypotheses are fulfilled for applying Lemma \ref{lemme:trou}, which gives the result.
\end{proof}

We will see in the next section that this geometric property of $\Omega_+$ can be exploited to show that the symmetrisation process used in the classical proof of the Rayleigh Conjecture in dimension $2,3$ loses too much information if $\Omega_-\neq \emptyset$.

\section{Improving Talenti's comparison principle}\label{sec:talenti}

The classical proof of the Rayleigh Conjecture in dimension $d=2,3$ relies on a symmetrisation technique. Roughly speaking, the Laplacian of the eigenfunction $u$ is symmetrised on the symmetrisations of $\Omega_+:=\{u>0\}$ and $\Omega_-:=\{u<0\}$ after what Talenti's comparison principle is astutely used, giving rise to an auxilary problem: the well-known two-ball problem. See \cite{ashbaugh-benguria} for details. Since symmetrisation lies at the center of this procedure, let us first recall the definition of signed Schwarz symmetrisation.

\begin{definition}\label{def:schwarz}
Let $\omega$ be a bounded open set and $f$ be a measurable function defined on $\omega$. The Schwarz symmetrisation $\omega^*$ of $\omega$ is any ball of same volume than $\omega$. The nonincreasing rearrangment of $f$ is the measurable function $f^*$ defined on $[0,|\omega|)$ to be the generalised inverse of the distribution function $\mu_{f}:(t\in\R\mapsto|\{x\in\omega:f(x)>t\}|)$ of $f$. In other words,
$$
\left\{
\begin{array}{ll}
f^*(s):=\mu_{f}^{[-1]}(s)=\inf\{\mu_{f}< s\}=\inf\{t:|\{f>t\}|< s\} & \forall s\in(0,|\omega|), \\
f^*(0):=ess\sup f.  
\end{array}
\right.
$$
Then, the signed Schwarz symmetrisation of $f$ is the measurable function $f^*$ defined on $\omega^*$ such that for all $x\in\omega^*$,
$$
f^*(x):=f^*(|B_{|x|}|),
$$
where $B_r$ denotes the ball of radius $r$ and of same center as $\omega^*$.
\end{definition}

\begin{remarque}
\begin{enumerate}
\item In the previous definition, we used the notation $f^*$ to designate two different functions, one defined on $[0,|\omega|)$, and the other defined on $\omega^*$. There will never be ambiguity in the following, precisely because the domains of the functions are not the same, hence there is no issue. Furthermore, a similar abuse of notation will be used in Definition \ref{def:talenti}.

\item Our definition of signed Schwarz symmetrisation follows the framework of \cite{kesavan}. Note that signed Schwarz symmetrisation differs from standard Schwarz symmetrisation (see \cite[Remark 1.1.2]{kesavan}). Indeed, classically, the Schwarz symmetrisation of some function $f$ corresponds to $|f|^*$. The advantage of our choice is that, even when $f$ changes sign, $f$ and $f^*$ remain equimeasurable.
\end{enumerate}
\end{remarque}

In 1981, Talenti realized that unsigned Schwarz symmetrisation were not adapted for dealing with signed functions such as the first eigenfunction of the bilaplacian. However, he did not resort to signed Schwarz symmetrisation, and rather introduced in \cite{talenti81} a special rearrangment in the following way.

\begin{definition}\label{def:talenti}
Let $\omega$ be and $f$ be a measurable function defined on $\omega$. We define for all $s\in[0,|\omega|)$
$$
f^\dagger(s)=f_+^*(s)-f_-^*(|\omega|-s).
$$
Then, the Talenti symmetrisation of $f$ is the function $f^\dagger$ defined on $\omega^*$ such that for all $x\in \omega^*$,
$$
f^\dagger(x):=f^\dagger(|B_{|x|}|).
$$
\end{definition}

\begin{remarque}
By convention, $f_+^*$ (resp. $f_-^*$) means $(f_+)^*$ (resp. $(f_-)^*$).
\end{remarque}

At this point, one shall wonder what is the relation between signed Schwarz symmetrisation and Talenti symmetrisation. Actually they coincide almost everywhere (see Lemma \ref{lemme:les symétrisations signées coincident} in appendix \ref{annexe:talenti}). Nevertheless, as we shall see, Talenti's formulation turns out to be sometimes convenient (see in particular Lemma \ref{lemme:symmetrisation de talenti}). This is probably why Talenti introduced it.

Let us now recall the plain comparison principle of Talenti (see \cite[Theorem 3.1.1]{kesavan}), which is the cornerstone of the classical proof of the Rayleigh Conjecture in dimension $d=2,3$.

\begin{theoreme}\label{thm:talenti}
Let $\omega$ be a bounded open set. Let $f\in L^2(\omega)$ and let $u\in H_0^1(\omega)$ solve
\begin{equation*}
\left\{\begin{array}{rcll}
-\Delta u & = & f & in\quad\omega,\\
u & = & 0 & on\quad\partial\omega.
\end{array}\right.
\end{equation*}
Let $v\in H_0^1(\omega^*)$ be the solution to
\begin{equation*}
\left\{\begin{array}{rcll}
-\Delta v & = & f^* & in\quad \omega^*,\\
v & = & 0 & on\quad\partial \omega^*.
\end{array}\right.
\end{equation*}
Then, if $u>0$ in $\omega$,
\begin{equation*}
v\geq u^*\qquad\text{a.e. in }\omega^*.
\end{equation*}
\end{theoreme}

\begin{remarque}
We mention that in his seminal paper \cite{talenti76}, Talenti uses the standard (unsigned) Schwarz symmetrisation, and hence ends up with a slightly weaker comparison principle than Theorem \ref{thm:talenti}. Indeed, one can deduce Theorem 1 of \cite{talenti76} from Theorem \ref{thm:talenti} combined with the maximum principle.
\end{remarque}

After these reminders, let us present the main result of this section. Its purpose is to refine Theorem \ref{thm:talenti} when $\omega$ satisfies the geometric property of $\Omega_+$, that is when $\omega$ has a hole.  Adapting the arguments proposed in \cite[section 3]{talenti76}, one obtains the following statement.

\begin{theoreme}\label{thm:talenti modifie}
Let $\omega$ be a bounded open set admitting a hole $T$. Let $f\in L^2(\omega)$ and let $u\in H_0^1(\omega)$ solve
\begin{equation*}
\left\{\begin{array}{rcll}
-\Delta u & = & f & in\quad\omega,\\
u & = & 0 & on\quad\partial\omega.
\end{array}\right.
\end{equation*}
Let $v\in H_0^1(\omega^*)$ be the solution to
\begin{equation*}
\left\{\begin{array}{rcll}
-\Delta v & = & f^* & in\quad \omega^*,\\
v & = & 0 & on\quad\partial \omega^*.
\end{array}\right.
\end{equation*}
Assume that $u\in C^2(\omega\cup\partial T)$ satisfies (\ref{hyp:u}) in $\omega$, that $\nabla u\neq0$ on $\partial T$, and that $u>0$ in $\omega$. Then,
\begin{equation*}
\kappa^2v\geq u^*\qquad\text{in }\omega^*,
\end{equation*}
where
$$
\kappa:=\frac{|\omega|^{\frac{d-1}{d}}}{(|\omega|+|T|)^{\frac{d-1}{d}}+|T|^{\frac{d-1}{d}}}<1.
$$
Moreover, in case of equality, $T$ and $\omega\cup\overline{T}$ are balls.
\end{theoreme}

\begin{remarque}
\begin{enumerate}
\item Observing that $\kappa<1$, we understand that our inequality $\kappa^2v\geq u^*$ is better than Talenti's one $v\geq u^*$ obtained with Theorem \ref{thm:talenti}.
\item Thanks to our regularity assumptions, the equality $\kappa^2v\geq u^*$ holds not only almost everywhere as in Theorem \ref{thm:talenti}, but everywhere in $\omega^*$.
\end{enumerate}
\end{remarque}

Before proving Theorem \ref{thm:talenti modifie}, we need a technical lemma. Indeed, we have to ensure that the geometric assumption made on $\omega=\{u>0\}$ in Theorem \ref{thm:talenti modifie} holds more generally on the upper level sets $\{u>t\}$ for all $0\leq t<\max_{\overline{\omega}}u$. Actually, this fact is not true in general without any assumption on $u$. However, due to our central assumption (\ref{hyp:u}), it will hold as stated below.

\begin{lemme}\label{lemme:trous}
Let $\omega$ be a bounded open set admitting a hole $T$ and $u\in C^2(\omega\cup\partial T)$ be positive and satisfying (\ref{hyp:u}) in $\omega$. Assume moreover that $u=0$ and $\nabla u\neq0$ on $\partial T$. Then, for all\/ $0\leq t<\sup_{\omega} u$, $\omega_t:=\{u>t\}$ has a hole $T_t\supseteq T$.
\end{lemme}

\begin{proof}
Since $u:\omega\to\R$ verifies (\ref{hyp:u}), it is a Morse function on $\omega\setminus\{u=\sup_\omega u\}$. Then, driven by the gradient near $\partial T\subseteq\partial\omega$, we will build as in Morse theory an open set $T_t$ which satisfies $T\subseteq T_t$, $T_t\cap\omega\subseteq\{u<t\}$, and $\partial T_t\subseteq\{u=t\}$. To do so, set $X=\frac{\nabla u}{|\nabla u|^2}$, well defined on $\omega\cup\partial T\setminus\{u=\sup_\omega u\}$, and let $\phi_t$ be the flow associated with $X$. We check that $u\circ\phi_{s_2}-u\circ\phi_{s_1}=s_2-s_1$, so that, $\forall x\in\partial T$, $u(\phi_s(x))=s$. Consequently, $T_t:=T\bigcup(\cup_{0\leq s<t}\phi_s(\partial T))$, which is open, satisfies the above requirements.

Then, immediately, $T\subseteq T_t$ and $T_t\cap\omega_t=\emptyset$. It remains to check that $\partial T_t\subseteq\partial\omega_t$. Let $x\in\partial T_t$ and $\epsilon >0$. Recall that $u(x)=t$ and that $\nabla u(x)\neq 0$ since (\ref{hyp:u}) is assumed. Then, because $u$ is $C^2$, we can write a Taylor expansion at $x$: for all $s>0$, set $y_s:=x+s\frac{\nabla u(x)}{|\nabla u(x)|}$, then
$$
u(y_s)=u(x)+\nabla u(x)\cdot(y_s-x)+o(|y_s-x|)=t+s|\nabla u(x)|+o(s).
$$
Consequently, if $s$ is small enough such that $y_s\in B_\epsilon(x)$ and that $o(s)\geq -s|\nabla u(x)|/2$, we get $u(y_s)\geq t+s|\nabla u(x)|/2>t$, or in other words $y_s\in\omega_t$. Therefore, $x\in\partial\omega_t$.
\end{proof}

\begin{remarque}\label{remarque: U}
Roughly speaking, the proof of Lemma \ref{lemme:trous} is based on the fact that the upper level sets of a function do not change topology between critical values. This clarifies the purpose of assumption (\ref{hyp:u}), which precisely rules out the possibility for the function to have critical values (except extremal ones). Thus, we understand that assumption (\ref{hyp:u}) made on the first eigenfunction $u$ of the optimal shape $\Omega$ in Theorem \ref{thm:faber-krahn U} shall be relaxed as long as it is ensured that the upper level sets of $u|_{\{u>0\}}$ and $-u|_{\{u<0\}}$ do not change topology. For instance, instead of (\ref{hyp:u}), one could ask $u$ to satisfy:

\noindent\emph{
For $v=u$ and $v=-u$, $\overline{\{v>t\}}$ is homeomorphic to $\overline{\{v>0\}}$ for each $0<t<\max\limits_{\overline{\Omega}} v$.
}

Actually, we believe that the points responsible for topology changes in the upper level sets are only \textbf{weak} saddle points (a weak saddle point is a critical point which is not a local \textbf{strict} extremum).
From this fact, one observes that

\noindent\emph{
As long as $u$ admits no weak saddle point in $\Omega$, $u$ has empty nodal set, and hence $\Omega$ is a ball.
}

Indeed, if $u$ had a nodal set, $\Omega_+$ would have a hole $T$ (Corollary \ref{corollaire:trou}). Hence, if moreover $u$ had no weak saddle point, the upper level sets $\{u>t\}$ for $0<t<\max_{\overline{\Omega}} u$ would also have holes containing $T$. Eventually, $\{u=\max_{\overline{\Omega}} u\}$ would have a hole containing $T$ (note that the definition of a hole actually makes sense not only for open sets). In particular, any point in $\{u=\max_{\overline{\Omega}} u\}$ would not be a strict maximum, hence it would be a weak saddle point, a contradiction with the assumption.
\end{remarque}

Thanks to Lemma \ref{lemme:trous}, it becomes possible to prove Theorem \ref{thm:talenti modifie}.

\begin{proof}[Proof of Theorem \ref{thm:talenti modifie}]
We follow the proof of Theorem \ref{thm:talenti} given in section 3 of \cite{talenti76}. Setting for all $0<t<\max_{\overline{\omega}} u$, $\omega_t:=\{u>t\}$, and $n_t$ the outward normal to the boundary of $\omega_t$ (note that $\partial\omega_t$ is regular due to assumption (\ref{hyp:u})), the divergence formula gives
$$
\int_{\omega_t}f=-\int_{\omega_t}\Delta u=-\int_{\partial\omega_t}\nabla u\cdot\vec{n_t}.
$$
But because of (\ref{hyp:u}), $n_t=-\frac{\nabla u}{|\nabla u|}$, so that
$$
\int_{\omega_t}f=\int_{\partial\omega_t}|\nabla u|.
$$
Now denote $\mu(t)=|\omega_t|$ the distribution function of $u$, and observe, thanks to Theorem 2.2.3 of \cite{kesavan}, that
$$
\mu'(t)=-\int_{\partial\omega_t}\frac{1}{|\nabla u|}.
$$
Moreover, thanks to Cauchy-Schwarz inequality,
\begin{equation*}
|\partial\omega_t|=\int_{\partial\omega_t}\sqrt{\frac{|\nabla u|}{|\nabla u|}}\leq\left(\int_{\partial\omega_t}\frac{1}{|\nabla u|}\int_{\partial\omega_t}|\nabla u|\right)^{1/2}=\sqrt{-\mu'(t)\int_{\omega_t}f}.
\end{equation*}
Hence,
\begin{equation}\label{eq: preuve talenti schwarz}
|\partial\omega_t|^2\leq-\mu'(t)\int_{\omega_t}f.
\end{equation}
At this point the classical proof uses the isoperimetric inequality to justify that
\begin{equation}\label{eq:inégalité isoperimetrique}
\frac{|\partial\omega_t|}{|\omega_t|^{\frac{d-1}{d}}}\geq \frac{|\partial\omega_t^*|}{|\omega_t^*|^{\frac{d-1}{d}}}=:C_d.
\end{equation}
But thanks to the geometric assumption, one is able to improve this inequality. Indeed, because $\omega$ has a hole $T$, Lemma \ref{lemme:trous} implies that $\omega_t$ has a hole $T_t$ containing $T$. In particular, due to the definition of a hole, $|\partial\left(\omega_t\cup\overline{T_t}\right)|=|\partial\omega_t|-|\partial T_t|$ and $|\omega_t\cup\overline{T_t}|=|\omega_t|+|T_t|$. Consequently, applying (\ref{eq:inégalité isoperimetrique}) to $\omega_t\cup\overline{T_t}$ rather than to $\omega_t$, one gets
\begin{equation*}
\frac{|\partial \omega_t|-|\partial T_t|}{(|\omega_t|+|T_t|)^{\frac{d-1}{d}}}\geq C_d.
\end{equation*}
Using once again (\ref{eq:inégalité isoperimetrique}) with $T_t$ this time, and then recalling that $T\subseteq T_t$, we find
\begin{equation*}
|\partial\omega_t|\geq C_d(|\omega_t|+|T_t|)^{\frac{d-1}{d}}+|\partial T_t|\geq C_d\left[(|\omega_t|+|T_t|)^{\frac{d-1}{d}}+|T_t|^{\frac{d-1}{d}}\right]\geq C_d\left[(\mu(t)+|T|)^{\frac{d-1}{d}}+|T|^{\frac{d-1}{d}}\right].
\end{equation*}
Plugging this into (\ref{eq: preuve talenti schwarz}),
\begin{equation*}
C_d^2\mu(t)^{\frac{2(d-1)}{d}}\left[\left(1+\frac{|T|}{\mu(t)}\right)^{\frac{d-1}{d}}+\left(\frac{|T|}{\mu(t)}\right)^{\frac{d-1}{d}}\right]^2\leq-\mu'(t)\int_{\omega_t}f.
\end{equation*}
Because $\mu(t)\leq|\omega|$, the previous inequality yields
\begin{equation*}
1\leq\frac{1}{\left[\left(1+\frac{|T|}{|\omega|}\right)^{\frac{d-1}{d}}+\left(\frac{|T|}{|\omega|}\right)^{\frac{d-1}{d}}\right]^2}\frac{-\mu'(t)}{C_d^2\mu(t)^{\frac{2(d-1)}{d}}}\int_{\omega_t}f=\kappa^2\frac{-\mu'(t)}{C_d^2\mu(t)^{\frac{2(d-1)}{d}}}\int_{\omega_t}f.
\end{equation*}
Now use Lemma \ref{lemme:symétrisation de Talenti croissante} in appendix \ref{annexe:talenti}:
$$
\int_{\omega_t}f=\int_{\omega_t} f|_{\omega_t}=\int_{\omega_t^*} \left(f|_{\omega_t}\right)^*\leq\int_{\omega_t^*} f^*|_{\omega_t^*}=\int_{\omega_t^*}f^*=\int_0^{\mu(t)}f^*.
$$
Consequently,
$$
1\leq\kappa^2\frac{-\mu'(t)}{C_d^2\mu(t)^{\frac{2(d-1)}{d}}}\int_0^{\mu(t)}f^*.
$$
Then we integrate both sides between $0$ and $t$ to find that
\begin{equation*}
t\leq\kappa^2\int_{\mu(t)}^{|\omega|}\frac{1}{C_d^2r^{\frac{2(d-1)}{d}}}\int_0^{r}f^*.
\end{equation*}
Applying this to $t=u^*(s)$, one gets, thanks to the fact that $u^*:[0,|\omega|)\to\R_+$ is the generalized inverse of the non increasing function $\mu:\R_+\to[0,|\omega|)$ (see section 2 \cite{talenti76}),
\begin{equation*}
u^*(s)\leq\kappa^2\int_{s}^{|\omega|}\frac{1}{C_d^2r^{\frac{2(d-1)}{d}}}\int_0^{r}f^*.
\end{equation*}
Finally, because $v$ is known to be radially symmetric, $-\Delta v=f^*$ in $\omega^*$ turns into an ordinary differential equation, the solution of which is explicit (see Step 4 of the proof of Theorem 3.1.1 in \cite{kesavan} for details), hence one is able to write
\begin{equation*}
u^*(x)\leq\kappa^2\int_{|B_{|x|}|}^{|\omega|}\frac{1}{C_d^2r^{\frac{2(d-1)}{d}}}\int_0^{r}f^*=\kappa^2v(x).
\end{equation*}

Regarding the equality case, note that if $u^*=\kappa^2 v$, then all the previous inequalities turn into equalities, and, in particular, thanks to the equality case in the isoperimetric inequality, for almost each $0<t<\max u$, $T_t$ and $\omega_t\cup\overline{T_t}$ are balls. Therefore, each $\omega_t$ is the set difference of two balls $B_{r_t}(x_t)\subseteq B_{R_t}(X_t)$, and
$$
\omega=\{u>0\}=\bigcup\limits_{t>0}\{u>t\}=\bigcup\limits_{t>0}\omega_t
$$
is the monotonic union of those differences. Up to extraction, assume that $x_t\to x$, $X_t\to X$, $r_t\to r$ and $R_t\to r$ when $t\to0$. The monotonicity of the union shows that $\omega\subseteq\{z:|z-x|\geq r,|z-X|\leq R\}$. Since $\omega$ is open, actually $\omega\subseteq B_R(X)\setminus \overline{B_r}(x)$. Due to the definition of $x$, $X$, $r$, and $R$, this inclusion turns into an equality, which implies that $T$ and $\omega\cup\overline{T}$ are balls.
\end{proof}

Thanks to our enhanced comparison principle, it is possible to bound from below the eigenvalue $\Gamma(\Omega)$ by an asymmetric two-ball problem. The next section is devoted to the derivation of this lower bound.

\section{Derivation of an asymmetric two-ball problem}\label{sec:probleme aux deux boules}

Proceeding as in \cite{ashbaugh-benguria}, but using our modified comparison principle, we get the following.

\begin{proposition}\label{prop:two-balls problem}
Let\/ $\Omega$ be a \reg\/ optimal shape for (\ref{eq:pb}). Let $a$ (resp. $b$) be the radius of $\Omega_+^*$ (resp. $\Omega_-^*$). Then, if the first eigenfunction on $\Omega$ satisfies (\ref{hyp:u}),
\begin{equation}\label{eq:pb deux boules}
\Gamma(\Omega)\geq\min_{v,w}\frac{\int_{B_a}(\Delta v)^2+\int_{B_b}(\Delta w)^2}{K(a,b)^4\int_{B_a}v^2+\int_{B_b} w^2}=:\mu(a,b),
\end{equation}
where
\begin{equation}\label{eq:K}
K(a,b)=\frac{a^{d-1}}{(a^d+b^d)^{\frac{d-1}{d}}+b^{d-1}},
\end{equation}
and the minimum is taken among radial functions $v\in H^2(B_a)$ and $w\in H^2(B_b)$ satisfying, for all $x\in\partial B_a$ and $y\in\partial B_b$,
$$
v(x)=w(y)=0, \quad a^{d-1}\partial_rv(x)=b^{d-1}\partial_rw(y), \quad \Delta v(x)+\Delta w(y)=0.
$$
Moreover, in case of equality in (\ref{eq:pb deux boules}), $\Omega$ is a ball.
\end{proposition}

\begin{proof}
Assume that $u$ changes sign otherwise there is no issue. As in \cite{talenti81} and \cite{ashbaugh-benguria} one symmetrises the sign-changing functions
$$
f:=-\Delta u, \qquad g:=\Delta u.
$$
Now the solutions of the Laplace equations in $B_a$ with data $f^*$, denoted $v$ and in $B_b$ with data $g^*$, say $w$, are kind of spherical rearrangements of $u_+$ and $u_-$:
\begin{equation*}
\left\{\begin{array}{rcll}
-\Delta v & = & f^*|_{B_a} & in\quad B_a, \\
v & = & 0 & on\quad\partial B_a,
\end{array}\right.
\qquad
\left\{\begin{array}{rcll}
-\Delta w & = & g^*|_{B_b} & in\quad B_b, \\
w & = & 0 & on\quad\partial B_b.
\end{array}\right.
\end{equation*}
We will compare $\Gamma(\Omega)$ with some quotient involving $v$ and $w$. For that, recall $f^*=f^\dagger$ and $g^*=g^\dagger$ almost everywhere (Lemma \ref{lemme:les symétrisations signées coincident}, appendix \ref{annexe:talenti}), and apply Lemma \ref{lemme:symmetrisation de talenti} with $p=2$, and $\omega_\pm=\Omega_\pm$. This is possible since $u$ cannot vanish on a set of positive measure due to the unique continuation principle \cite{pederson,protter}, which justifies that assumption $|\Omega_+|+|\Omega_-|=|\Omega|$ of Lemma \ref{lemme:symmetrisation de talenti} is fullfield. We find
\begin{equation}\label{eq:numerateur pb deux boules}
\int_\Omega(\Delta u)^2=\int_{B_a}\left(\Delta v\right)^2+\int_{B_b}\left(\Delta w\right)^2.
\end{equation}
For a later use, we need also to apply Lemma \ref{lemme:symmetrisation de talenti} with $p=1$, which gives
\begin{equation}\label{eq:condition de bord pb deux boules}
\int_\Omega\Delta u=\int_{B_a}\Delta v-\int_{B_b}\Delta w.
\end{equation}
Thanks to (\ref{eq:numerateur pb deux boules}), the numerator in the Rayleigh quotient $\Gamma(\Omega)$ shall be written in terms of the $L^2$ norms of $\Delta v$ and $\Delta w$. At this stage, we would like to estimate the denominator in the Rayleigh quotient $\Gamma(\Omega)$ in terms of the $L^2$ norms of $v$ and $w$. In the classical proof of the Rayleigh Conjecture, noticing that $\int_{\Omega}u^2=\int_{\Omega_+} u_+^2+\int_{\Omega_-} u_-^2=\int_{B_a}(u_+^*)^2+\int_{B_b}(u_-^*)^2$, this is done using Theorem \ref{thm:talenti}.

More precisely, on the one hand, thanks to Lemma \ref{lemme:symétrisation de Talenti croissante}, $-\Delta w=g^*|_{B_b}\geq \left(g|_{\Omega_-}\right)^*$ in $B_b$, and on the other hand, $-\Delta u_-=g|_{\Omega_-}$ in $\Omega_-$. Then, Theorem \ref{thm:talenti} combined with the maximum principle yields $w\geq u_-^*$ in $B_b$. Analogously, we get $v\geq u_+^*$ in $B_a$. But here this last inequality can be improved.

Indeed, with Corollary \ref{corollaire:trou}, we know that $\Omega_+$ has a hole $T$ containing $\Omega_-$. Moreover, as (\ref{hyp:u}) has been assumed to hold in $\Omega$, clearly it holds in $\Omega_+$ and even $\nabla u\neq 0$ on $\partial T$. Therefore, we apply Theorem \ref{thm:talenti modifie} with $\omega=\Omega_+$ and combine it as before with the maximum principle to obtain
\begin{equation*}
u_+^*\leq \left[\frac{|\Omega_+|^{\frac{d-1}{d}}}{\left(|\Omega_+|+|T|\right)^{\frac{d-1}{d}}+|T|^{\frac{d-1}{d}}}\right]^2v\leq \left[\frac{|\Omega_+|^{\frac{d-1}{d}}}{\left(|\Omega_+|+|\Omega_-|\right)^{\frac{d-1}{d}}+|\Omega_-|^{\frac{d-1}{d}}}\right]^2v.
\end{equation*}
Since $|\Omega_+|=|B_1|a^d$ and $|\Omega_-|=|B_1|b^d$, we get
\begin{equation*}
u_+^* \leq \left[\frac{a^{d-1}}{\left(a^d+b^d\right)^{\frac{d-1}{d}}+b^{d-1}}\right]^2 v= K(a,b)^2v.
\end{equation*}
To sum up,
\begin{align*}
&\int_\Omega(\Delta u)^2=\int_{B_a}(\Delta v)^2+\int_{B_b}(\Delta v)^2,\\
&\int_\Omega u^2 = \int_{B_a}(u_+^*)^2+\int_{B_b}(u_-^*)^2\leq K(a,b)^4\int_{B_a} v^2+\int_{B_b} w^2.
\end{align*}
Therefore,
\begin{equation*}
\Gamma(\Omega)\geq \frac{\int_{B_a}(\Delta v)^2+\int_{B_b}(\Delta w)^2}{K(a,b)^4\int_{B_a}v^2+\int_{B_b} w^2}.
\end{equation*}
Before concluding, we emphasize that for every $x\in B_a$ and every $y\in B_b$,
\begin{enumerate}
\item $v(x)=0=w(y)$;
\item due to (\ref{eq:condition de bord pb deux boules}),
$$
|\partial B_a|\partial_rv(x)-|\partial B_b|\partial_rw(y)=\int_{\partial B_a}\partial_nv-\int_{\partial B_b}\partial_n w=\int_{B_a}\Delta v-\int_{B_b}\Delta w=\int_\Omega \Delta u=0;
$$
\item  $\Delta v(x)+\Delta w(y)=f^\dagger(|B_a|)+(-f)^\dagger(|B_b|)=f_+^*(|B_a|)-f_-^*(|B_b|)+f_-^*(|B_b|)-f_+^*(|B_a|)=0$.
\end{enumerate}
Therefore the couple $(v,w)$ is admissible and hence $\Gamma(\Omega)\geq\mu(a,b)$. Eventually, note that $\mu(a,b)$ is well-posed (see \cite[Appendix 2]{ashbaugh-benguria}).

Regarding the rigidity result, we observe that the equality case implies that there are equalities both in Talenti's comparison principle $w\geq u_-^*$ and in our enhanced comparison principle $K(a,b)^2v\geq u_+^*$. On the one hand, due to the equality case in Theorem \ref{thm:talenti modifie} this means that $T$ and $\Omega_+\cup\overline{T}$ are balls. On the other hand, due to \cite[section 3.2]{kesavan}, this shows that $\Omega_-$ is a ball (contained in $T$). But as $\Omega$ is open and connected (see the proof of Corollary \ref{corollaire:trou}), the only possibility for $\Omega$ is to be a ball.
\end{proof}

\section{Solving the asymmetric two-ball problem}\label{sec:resolution deux boules}

As noted in \cite{ashbaugh-benguria}, the conjecture will hold if one is able to show that
$$
\mu(0,1)=\mu(1,0)=\min_{\substack{a,b\in[0,1] \\ a^d+b^d=1}}\mu(a,b).
$$
Indeed, it is always possible to assume that the radius of $\Omega^*$ is equal to $1$. Then, thanks to Proposition \ref{prop:two-balls problem} (and to the unique continuation principle \cite{pederson,protter} yielding $a^d+b^d=1$),
\begin{equation*}
\Gamma(\Omega)\geq\mu(a,b)\geq\min_{\substack{a,b\in[0,1] \\ a^d+b^d=1}}\mu(a,b)\geq\mu(0,1)=\Gamma(B_1).
\end{equation*}

This is exactly the statement of the next result. Let us mention at this point that it seems difficult to avoid technicalities, since the value of $\mu(a,b)$ has to be clarified, leading to explicit computations. This issue was already present in the resolution of the symmetric two-ball problem, and it will be intensified here, the asymmetry complicating the analysis. In particular, in order to understand the assumptions of the theorem, we need to introduce several notations. In the following, $J_\nu$ and $I_\nu$ are the Bessel and modified Bessel functions of order $\nu$. The function $f_\nu$ is defined in accordance with the formula
\begin{equation}
f_\nu(r)=r^{d-1}\left[\frac{J_{\nu+1}(r)}{J_\nu(r)}+\frac{I_{\nu+1}(r)}{I_\nu(r)}\right].
\end{equation}
The first positive zero of $J_\nu$ is denoted $j_\nu$ and $k_\nu$ is the first positive zero of $f_\nu$. For each $a\in[0,1]$, the number $b\in[0,1]$ such that $a^d+b^d=1$ is denoted $b(a)$. Then, $K(a)$ is a simplified notation for $K(a,b(a))$ where $K(a,b)$ is defined in (\ref{eq:K}). We define $a_I$ to be the value of $a$ such that $b(a)=j_\nu/k_\nu$. On the other hand, $a_S$ is the value of $a$ when $aK(a)=j_\nu/k_\nu$. Eventually, we set
\begin{equation}
\begin{split}
T_1(a) = & (aK(a))^dK'(a)/K(a),\\
g_\nu(r) = & \frac{1}{r^{d-1}}\left[\frac{J_{\nu+1}}{J_\nu}(r)-\frac{I_{\nu+1}}{I_\nu}(r)\right],\\
G_1(a) = & [aK'(a)+K(a)](k_\nu K(a))^{d-1}g_\nu(k_\nu aK(a)),\\
G_2(a) = & k_\nu^{d-1}g_\nu(k_\nu b(a)).
\end{split}
\end{equation}
Using the above notations, it becomes possible to state the main result of this section.

\begin{theoreme}\label{thm:pb aux 2 boules}
Let $d\geq 4$, set $\nu=d/2-1$, and assume that
\begin{equation}\label{eq:condition necessaire kv jv}
2\frac{j_\nu^d}{k_\nu^d}+\frac{j_\nu}{k_\nu}>1.
\end{equation}
Assume moreover that
\begin{enumerate}
\item\label{it:hypothese Fv<0 sur AI} There exists a sequence $0=x_0<x_1<...<x_n<x_{n+1}=a_I$ such that, for each $i\in]0,n]$,
\begin{equation}\label{eq:Fv en xi}
f_\nu(k_\nu K(x_{i+1})x_{i+1})+K(x_i)^df_\nu(k_\nu b(x_i)) \leq 0,
\end{equation}
and
\begin{equation}\label{eq:Gv en xi}
G_1(x_1)+G_2(x_0)\leq0.
\end{equation}
\item\label{it:hypothese Fv<0 sur AS} There exists a sequence $a_S=y_0<y_1<...<y_m<y_{m+1}=1$ such that for each $i\in[0,m[$,
\begin{equation}\label{eq:Fv en yi}
f_\nu(k_\nu K(y_{i+1})y_{i+1})+f_\nu(k_\nu b(y_i))\leq0,
\end{equation}
and
\begin{align}\label{eq:Fv'>0 sur ]ym,1[}
2T_1(y_m) -f_\nu(k_vb(y_m))|g_\nu(k_vK(y_m)y_m)|\left[\frac{(d-1)}{y_m^d}\left(1+\frac{1}{b(y_m)}\right)+1\right] -[f_\nu g_\nu](k_vb(y_m))\geq0.
\end{align}
\end{enumerate}
Then,
\begin{equation*}
\min_{\substack{a,b\in[0,1] \\ a^d+b^d=1}}\mu(a,b)=\mu(0,1)=\mu(1,0).
\end{equation*}
\end{theoreme}

\begin{remarque}\label{remarque:pb aux 2 boules}
Let us point out that each of the assumptions of Theorem \ref{thm:pb aux 2 boules} might be checked numerically. This has been done for $d=4,5,6,7,8,9$ in appendix section \ref{sec:numerique}.
\end{remarque}

\begin{proof}
Fix for the moment an admissible couple $(a,b)$. As explained in \cite[Appendix 2]{ashbaugh-benguria}, the minimum $\mu(a,b)$ is obtained for two $H^2$ radial functions $v=v(r)$ and $w=w(r)$ satisfying
\begin{equation}\label{eq:minimisation pb deux boules, equation aux valeurs propres}
\Delta^2 v=K(a,b)^4\mu(a,b) v \text{ in } B_a,
\qquad
\Delta^2 w=\mu(a,b) w \text{ in } B_b,
\end{equation}
and the boundary conditions, for all $x\in \partial B_a$ and $y\in \partial B_b$,
\begin{equation}\label{eq:minimisation pb deux boules, condition de bord}
v(x)=w(y)=0, \qquad a^{d-1}\partial_rv(x)=b^{d-1}\partial_rw(y), \qquad \Delta v(x)+\Delta w(y)=0.
\end{equation}
Because $v,w$ are radial, (\ref{eq:minimisation pb deux boules, equation aux valeurs propres}) is actually an ODE of which any solution is of the form
\begin{align*}
& v(r)=[AJ_\nu(kKr)+BI_\nu(kKr)]r^{-\nu}, \\
& w(r)=[CJ_\nu(kr)+DI_\nu(kr)]r^{-\nu},
\end{align*}
where $\nu=d/2-1$, $k=k(a)=k(a,b):=\mu(a,b)^{1/4}$ and $K=K(a)=K(a,b)=\frac{a^{d-1}}{1+b^{d-1}}$. Using the previous expressions and the relations $J_{\nu}'(r)=\frac{\nu J_\nu(r)}{r}-J_{\nu+1}(r)$ and $I_{\nu}'(r)=\frac{\nu I_\nu(r)}{r}+I_{\nu+1}(r)$, one finds
\begin{align*}
& \partial_rv(r)=[-AJ_{\nu+1}(kKr)+BI_{\nu+1}(kKr)]kKr^{-\nu}, \\
& \partial_rw(r)=[-CJ_{\nu+1}(kr)+DI_{\nu+1}(kr)]kr^{-\nu}.
\end{align*}
Using this time the relation $J_{\nu+1}'(r)=J_\nu(r)-\frac{\nu+1}{r}J_{\nu+1}(r)$ and $I_{\nu+1}'(r)=I_\nu(r)-\frac{\nu+1}{r}I_{\nu+1}(r)$, we obtain
\begin{align*}
& \partial_r^2v(r)=\left(A\left[\frac{(2\nu+1)J_{\nu+1}(kKr)}{kKr}-J_\nu(kKr)\right]-B\left[\frac{(2\nu+1)I_{\nu+1}(kKr)}{kKr}-I_\nu(kKr)\right]\right)(kK)^2r^{-\nu}, \\
& \partial_r^2w(r)=\left(C\left[\frac{(2\nu+1)J_{\nu+1}(kr)}{kr}-J_\nu(kr)\right]-D\left[\frac{(2\nu+1)I_{\nu+1}(kr)}{kr}-I_\nu(kr)\right]\right)k^2r^{-\nu}.
\end{align*}
Then,
\begin{align*}
& \Delta v(r)=\left(-AJ_\nu(kKr)+BI_\nu(kKr)\right)(kK)^2r^{-\nu}, \\
& \Delta w(r)=\left(-CJ_\nu(kr)+DI_\nu(kr)\right)k^2r^{-\nu}.
\end{align*}
Eventually, (\ref{eq:minimisation pb deux boules, condition de bord}) yields a system of four linear equations involving $(A,B,C,D)$, that has only the solution $(0,0,0,0)$ unless the determinant of the system is $0$. Consequently, for the solution not to be trivial, one gets at the end of the day
\begin{align*}
0 = & 2a^{d-1}[I_\nu(kKa)I_\nu(kb)J_{\nu+1}(kKa)J_\nu(kb) + J_\nu(kKa)I_\nu(kb)I_{\nu+1}(kKa)J_\nu(kb)] \\
+ & 2Kb^{d-1}[I_\nu(kKa)I_\nu(kb)J_{\nu+1}(kb)J_\nu(kKa) + J_\nu(kKa)I_\nu(kb)J_{\nu+1}(kb)I_\nu(kKa)],
\end{align*}
which is equivalent to
\begin{equation*}
0=f_\nu(kKa)+K^df_\nu(kb),
\end{equation*}
where
\begin{equation*}
f_\nu(r)=r^{d-1}\left[\frac{J_{\nu+1}(r)}{J_\nu(r)}+\frac{I_{\nu+1}(r)}{I_\nu(r)}\right].
\end{equation*}

Therefore $k(a)$ is the first positive zero of $h_\nu:r\mapsto f_\nu(rKa)+K^df_\nu(rb)$. Hence, to conclude the proof, it remains to show that $k_\nu=k(0)=k(1)\leq k(a)$ for all $a\in[0,1]$. As shown in \cite{ashbaugh-benguria}, $f_\nu$ has some useful properties: $f_\nu(0)=0$; $f_\nu$ has simple poles at $j_{\nu,i}, i\in\N^*$ (the positive zeros of $J_\nu$); $f_\nu$ is continuous and increasing as soon as it is defined. Consequently $f_\nu(j_{\nu,i}^-)=+\infty$ and $f_\nu(j_{\nu,i}^+)=-\infty$. As explained in \cite{ashbaugh-benguria-laugesen}, it is useful to study the function of two variables $F_\nu(k,a)=f_\nu(kK(a)a)+K(a)^df_\nu(kb)$. The properties of $f_\nu$ quoted above justify that $F_\nu$ is increasing in its first variable and admits poles on the arcs $\alpha_i$ of equation $kKa=j_{\nu,i}$ and $\beta_i$ of equation $kb=j_{\nu,i}$. Then, $k(a)$ gives rise to an arc $\kappa=(k(a),a)$ joining the horizontal line $\{a=0\}$ to the horizontal line $\{a=1\}$ and lying in a connected component of the complement of the arcs $\alpha_i$ and $\beta_i$. More precisely, because $k$ remains between the first and the second pole of $h_\nu$, $\kappa$ is trapped between $\alpha_1$ and $\beta_1$. Thus one has the qualitative properties summarised in Figure \ref{fig:Fv}.

\begin{figure}[h!]
\centering
\includegraphics[width=\textwidth]{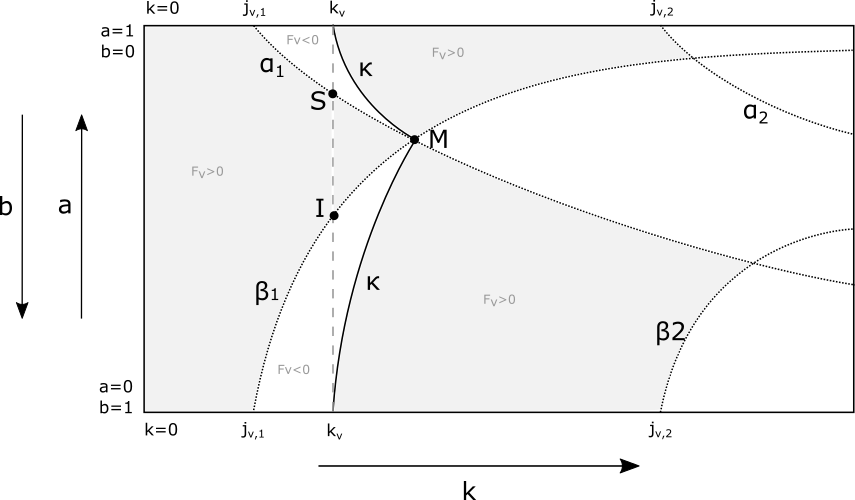}
\caption{Qualitative properties of the function $F_\nu$ in the $(k,a)$-plane.}
\label{fig:Fv}
\end{figure}

An interesting point in the $(k,a)$-plane is the point $M$ where $\alpha_1$ and $\beta_1$ intersect, because the previous discussion showed that $M\in\kappa$. At $M$, one has the relation $Ka=b$, or in other words $2b^d+b-1=0$. Note that the polynomial $P=2X^d+X-1$ satisfies $P(0)=-1$, $P(1)=2$ and is increasing between $0$ and $1$. If we note $b_M$ the root of $P$ between $0$ and $1$, and $a_M$ the corresponding value for $a$, then $k_M:=k(a_M)=\frac{j_{\nu,1}}{b_M}$ since $(k_M,a_M)\in\alpha_1$. As in \cite{ashbaugh-benguria}, one wonders whether $k_\nu<k_M$. To see this, one has to check if $b_M<\frac{j_{\nu,1}}{k_\nu}$, that is if $P(\frac{j_{\nu,1}}{k_\nu})>0$. Since by definition $j_\nu=j_{\nu,1}$, this condition turns into the next equation.
\begin{equation*}
2\frac{j_\nu^d}{k_\nu^d}+\frac{j_\nu}{k_\nu}>1.
\end{equation*}
This necessary condition for $k_\nu$ to be the minimum of $k$ is the couterpart of equation (41) of \cite{ashbaugh-benguria}, which failed for $d>3$. In our situation, it is assumed that (\ref{eq:condition necessaire kv jv}) is fulfilled (see also Table \ref{tab:condition nécessaire}). This gives credence to Figure \ref{fig:Fv}, where $M$ is at the right of the vertical dashed line $\{k=k_\nu\}$.

Obviously, the fact that $k_\nu<k_M$ holds is not sufficient to claim that $k_\nu$ minimises $k$. As suggested by Figure \ref{fig:Fv}, one checks that each of $\alpha_1$ and $\beta_1$ intersect exactly once the line $\{k=k_\nu\}$. We denote by $S$ and $I$ the intersection points, and by $(k_S,a_S,b_I)$ and $(k_I,a_I,b_S)$ the coordinates of $S$ and $I$ respectively. These coordinates are characterised by the relations
\begin{equation}\label{eq:equations en I et S}
b_S=\frac{j_\nu}{k_\nu}, \qquad K(a_S)a_S=\frac{j_\nu}{k_\nu}.
\end{equation}
It is then clear that for $a\in[a_I,a_S]$, $k(a)>k_\nu$ (see Figure \ref{fig:Fv}). Hence it remains to show that for $a\in[0,a_I[\cup]a_S,1]=:A$ we have $k(a)\geq k_\nu$. To do so, thanks to the fact that $F_\nu$ is increasing in the variable $k$, it is enough to prove that $\forall a\in A$,
\begin{equation}\label{eq:Fv<0}
F_\nu(k_\nu,a)\leq 0.
\end{equation}
The very technical and laboured proof of this inequality is provided by Theorem \ref{thm:Fv<0} in appendix \ref{annexe:pb 2 boules}, which relies on equations (\ref{eq:Fv en xi}), (\ref{eq:Gv en xi}), (\ref{eq:Fv en yi}), and (\ref{eq:Fv'>0 sur ]ym,1[}).
\end{proof}

\section{Proof of Theorem \ref{thm:faber-krahn U} and final remarks}\label{sec:preuve theoreme}

Thanks to the analysis performed in sections \ref{sec:probleme aux deux boules} and \ref{sec:resolution deux boules}, the proof of Theorem \ref{thm:faber-krahn U} is rather immediate.

\begin{proof}[Proof of Theorem \ref{thm:faber-krahn U}]
Up to a dilation, assume that $|\Omega|=|B_1|$. Then, thanks to Proposition \ref{prop:two-balls problem}, we know that $\Gamma(\Omega)\geq \mu(a,b)$ (where $|B_a|=|\Omega_+|$ and $|B_b|=|\Omega_-|$). Moreover, due to the unique continuation principle \cite{pederson,protter}, the first eigenfunction does not vanish on a set of positive measure, hence $a^d+b^d=1$. Then, thanks to Theorem \ref{thm:pb aux 2 boules}, and to Remark \ref{remarque:pb aux 2 boules} we know that $\mu(a,b)\geq\mu(0,1)=\Gamma(B_1)$ whenever $4\leq d \leq 9$. Combining the inequalities, we get
$$
\Gamma(\Omega)\geq\Gamma(B_1).
$$
Moreover, in case of an equality, we have $\Gamma(\Omega)=\mu(a,b)$, and we conclude thanks to the equality case in Proposition \ref{prop:two-balls problem} that $\Omega$ is a ball.
\end{proof}

Theorem \ref{thm:faber-krahn U} mainly relies on the assumption (\ref{hyp:u}). Hence it is essential to understand to which extent (\ref{hyp:u}) has a chance to hold. For instance, one could seek for geometric conditions on $\Omega$ yielding (\ref{hyp:u}).
The easiest such condition one shall think about is maybe convexity.
Indeed, it is known that for second order operators, convexity often implies a unique critical point (hence a unique critical value) for the solutions of some particular equations (see \cite{grossi,garcia-melian} and also the unmissable \cite[chapter III]{kawohl}).
However, it is clear that convexity is not enough for showing that the first eigenfunction admits a unique critical value in the case of fourth order equations. To see this, consider classical counterexamples, such as domains with corners, in which it is known that the eigenfunctions of the bilaplacian oscillate \cite{coffman}, and consequently do not satisfy (\ref{hyp:u}).
Nevertheless, against these examples, one shall put forward that the shapes involved are not critical and hence not optimal. Therefore, one is led to wonder whether a first eigenfunction of a critical convex shape satisfies (\ref{hyp:u}).

At this time, we were not able to answer this central issue, and could only argue that if $u$ is a first eigenfunction on a $C^\infty$ strictly convex critical shape $\Omega\subseteq \R^2$, then $\Delta u+\sqrt{\Gamma(\Omega)}u$ admits only one critical point. This is a consequence of \cite[Proposition 4]{leylekian} and \cite[Theorem 1.2]{regibus-grossi-mukherjee}. Indeed, due to the convexity of $\Omega$, $\partial\Omega$ is connected, and due to its criticality, up to a sign change, $\frac{\Delta}{\sqrt{\Gamma(\Omega)}}u=\sqrt{\frac{2}{|\Omega|}}$ on $\partial\Omega$. Then, thanks to \cite[Proposition 4]{leylekian}, the function $z_u:=\frac{\Delta}{\sqrt{\Gamma(\Omega)}}+u-\sqrt{\frac{2}{|\Omega|}}$ satisfies
$$
\left\{
\begin{array}{rcll}
\Delta z_u & = &\sqrt{\Gamma(\Omega)}\left(z_u+ \sqrt{\frac{2}{|\Omega|}}\right) & in \quad\Omega,\\
z_u & < & 0 & in \quad\Omega,\\
z_u & = & 0 & on \quad\partial\Omega,
\end{array}
\right.
$$
and we conclude that $-z_u$ is a positive semi-stable solution. Hence, it has a unique critical point thanks to \cite[Theorem 1.2]{regibus-grossi-mukherjee}.

Besides the geometric conditions on $\Omega$ that would yield (\ref{hyp:u}), another interesting question is the sensitivity of (\ref{hyp:u}) under perturbations of $\Omega$. Actually, for small $C^4$ perturbations of the form $(I+\theta)\Omega=:\Omega_\theta$, it is possible to build a first eigenfunction $u_\theta$ on $\Omega_\theta$ varying continuously with $\theta$, in the sense that $\theta\in C^4(\R^d,\R^d)\mapsto  u_\theta\circ(I+\theta)\in C^{3,\gamma}(\overline{\Omega})$ is continuous in a neighbourhood of $\theta=0$. On the other hand, let us define the map $\#:C^2(\overline{\Omega})\to\N\cup\{\infty\}$, such that $\#v$ is the number of critical points of $v$ in $\Omega$. Then, one can prove that $\#$ is upper-semi continuous at any point $v$ having no degenerate critical point and such that $\Delta v\neq 0$ on $\partial\Omega$. Combining those two results, we obtain that if $\Omega$ is a critical shape, if $u_{0}$ has only one critical point, and if this critical point is nondegenerate, then also $u_\theta$ has only one critical point for all small enough $\theta$. In particular, $u_\theta$ satisfies (\ref{hyp:u}) for all small enough $\theta$. This applies to the case $\Omega=B$, and shows, thanks to Theorem \ref{thm:faber-krahn U}, that small $C^4$ nontrivial perturbations of the ball cannot be optimal, slightly generalising a known result (see \cite[Theorem 6.29]{gazzola-grunau-sweers}).

As a final remark, let us mention that, surprinsgly, assumption (\ref{hyp:u}) is similar to the condition given in Remark 3 of \cite{cianchi} for the one-dimensional nonnincreasing rearrangment to reduce the $L^p$ norm of the second derivative. Indeed, at least in dimension one, it is known that what prevents Schwarz symmetrisation from preserving positive $H_0^2$ functions are the possible critical values of those functions. Following this idea, it would be interesting to study to which extent assumption (\ref{hyp:u}) yields a P\'olya-Szegö type inequality for second-order derivatives.



\appendix

\begin{center}
\Large Appendices
\end{center}

\section{Holes}\label{annexe:trou}

This appendix provides a proof for the geometric Lemma \ref{lemme:trou}. We recall that Lemma \ref{lemme:trou} claims that if a connected open set $\omega$ shall be divided into two open sets of which the boundaries do not meet simultaneously $\partial\omega$, then one of those two sets must have a hole.


\begin{proof}[Proof of Lemma \ref{lemme:trou}]
We claim first that if $\gamma_+$ and $\gamma_-$ are connected components respectively of $\partial\omega_+$ and $\partial\omega_-$, such that $\gamma_+\cap\gamma_-\neq\emptyset$, then $\gamma_+=\gamma_-$. Indeed, otherwise $\gamma_+\cap\gamma_-$ would be a proper subset of $\gamma_+\cup\gamma_-$, hence its (relative) boundary would be nonnempty. But any point $p$ of such boundary is not only in $\gamma_+\cap\gamma_-$, but also in $\partial\omega$ (otherwise, using the regularity of $\omega_\pm$ and looking at a neighbourhood of $p$, we would obtain a contradiction). To sum up, there would exist a point in $\partial\omega_+\cap\partial\omega_-\cap\partial\omega$, a contradiction. This means that the connected components of $\partial\omega_+$ and $\partial\omega_-$ either coincide or are disjoint.

But, necessarily, there exists a connected component of $\partial\omega_+$ that intersects a connected component of $\partial\omega_-$, otherwise $\partial\omega_+\cap\partial\omega_-$ would be empty, and this would in turn mean that $\overline{\omega}$ is not connected (as it would be the disjoint union of the closed sets $\overline{\omega_+}$ and $\overline{\omega_-}$), and hence that $\omega$ is not connected (recall that the closure of a connected set is connected), which is impossible. This, and the previous discussion show that one is always able to find a connected component simultaneously of $\partial\omega_+$ and $\partial\omega_-$.

Let us denote $\Theta$ the set of unions of connected components shared by $\partial\omega_+$ and $\partial\omega_-$. The arguments above show that $\Theta$ is nonempty.
In the following, we will construct a maximal element $\tilde{\theta}$ of $\Theta$ in the sense that $\tilde{\theta}$ will enclose any other element of $\Theta$.
Indeed, as $\omega_\pm$ is regular relatively with respect to $\omega$, any element of $\Theta$ (which belongs to $\omega$ due to $\partial\omega\cap\partial\omega_+\cap\partial\omega_-=\emptyset$) form a compact $C^1$ hypersurface of $\R^d$.
Therefore, any connected component of an element of $\Theta$ is a $C^1$ compact connected hypersurface of $\R^d$.
Thanks to Jordan-Brouwer Separation Theorem for $C^1$ hypersurfaces \cite{lima}, such a connected component is then the boundary of a bounded connected open set.
Hence, sending any element $\theta\in\Theta$ to the union of the open sets delimeted by the connected components of $\theta$, one builds a well-defined (in general noninjective) map $\phi:\Theta\to\mathcal{T}(\R^d)$ (where $\mathcal{T}(\R^d)$ is the topology of $\R^d$).
The map $\phi$ has the properties that for any $\theta\in\Theta$, $\theta\subseteq\overline{\phi(\theta)}$ and $\partial\phi(\theta)\subseteq\theta$. Moreover, if $\theta\in\Theta$ is connected, $\phi(\theta)$ is also connected.
Then, define
$$
\tilde{T}:=\bigcup_{\theta\in\Theta}\phi(\theta),\quad\text{and}\quad \tilde{\theta}:=\partial\tilde{T}.
$$
One can show that:
\begin{enumerate}
\item\label{it:frontière de T dans Theta} $\tilde{\theta}\in\Theta$, in particular $\tilde{\theta}$ is $C^1$ regular.
\item\label{it:maximalité de T} For any $\theta\in\Theta$, $\theta\subseteq\overline{\tilde{T}}$.
\item\label{it:régularité tilde T} $\tilde{T}$ is open and $C^1$ regular.
\item\label{it:connexité de la frontière de tilde T} The connected components of $\tilde{T}$ have connected boundary.
\end{enumerate}
At this point, let us prove that $\tilde{T}^c$ intersects \textbf{at most} one of $\omega_+$ and $\omega_-$. Indeed, assume by contradiction that $\tilde{T}^c$ intersects both $\omega_+$ and $\omega_-$. Then, $u_+:=\omega_+\setminus\overline{\tilde{T}}$ and $u_-:=\omega_-\setminus\overline{\tilde{T}}$ are nonempty open sets, and moreover their closures intersect, otherwise the closure of $\omega\setminus\overline{\tilde{T}}$ would be disconnected as the disjoint union of two closed set. Then, $\omega\setminus\overline{\tilde{T}}$ would in turn be disconnected. But as $\omega$ is connected, as $\tilde{T}$ satisfies \ref{it:régularité tilde T}, \ref{it:connexité de la frontière de tilde T}, and as $\partial\tilde{T}\subseteq \omega$ (by \ref{it:frontière de T dans Theta} and because $\partial\omega\cap\partial\omega_+\cap\partial\omega_-=\emptyset$), this cannot occur due to Lemma \ref{lemme:connexe privé d'un ensemble ne recontrant pas sa frontiere} below. Therefore, $\overline{u_+}$ and $\overline{u_-}$ intersect, meaning that their boundaries intersect (as $u_+$ and $u_-$ remain disjoint). Let $\gamma$ be a connected component of $\partial u_+\cap\partial u_-$. Then, $\gamma\in\Theta$ by \ref{it:frontière de T dans Theta}, hence $\gamma\subseteq\overline{\tilde{T}}$ by \ref{it:maximalité de T}. We show that this implies $\tilde{T}\cap u_+\neq\emptyset$ or $\tilde{T}\cap u_-\neq\emptyset$, leading to a contradiction. Indeed, if $p\in\gamma$, by regularity of $\gamma\in\Theta$, there is a small ball $B$ centered at $p$ such that $\gamma$ divides $B$ into exactly two parts. Since $\gamma\subseteq\partial u_+\cap\partial u_-$ and since $u_+$ and $u_-$ are disjoint, those two parts correspond to $B\cap u_+$ and $B\cap u_-$. On the other hand, as $\gamma\subseteq\overline{\tilde{T}}$, $B$ meets $\tilde{T}$, and $\emptyset\neq B\cap\tilde{T}\subseteq (u_+\cap\tilde{T})\cup (u_-\cap\tilde{T})\cup\gamma$. In particular, since $B\cap\tilde{T}$ is open, either $u_+\cap\tilde{T}\neq\emptyset$ or $u_-\cap\tilde{T}\neq\emptyset$.

We have shown that $\tilde{T}^c$ intersects at most one of $\omega_+$ and $\omega_-$. On the other hand, $\tilde{T}^c$ intersects \textbf{at least} one of $\omega_+$ and $\omega_-$, otherwise $\omega\subseteq\tilde{T}$, hence $\tilde{\theta}\subseteq\partial\omega$. Then $\partial\omega\cap\partial\omega_+\cap\partial\omega_-=\emptyset$ would not hold. Therefore, $\tilde{T}^c$ intersects \textbf{exactly} one of $\omega_+$ and $\omega_-$. Changing the notations if needed, one can assume that $\tilde{T}^c$ intersects $\omega_+$, so that $\tilde{T}$ intersects $\omega_-$. Now set $T:=\tilde{T}\setminus\overline{\omega_+}$. Because $T$ is disjoint from $\omega_+$ and because \mbox{$\partial T\subseteq\partial\tilde{T}\cup\partial\omega_+\subseteq\partial\omega_+$}, we see that $T$ is a hole for $\omega_+$.

It remains only to check that $T$ contains $\omega_-$. But this comes from the fact that $\omega_-\subseteq\tilde{T}$ and from the definition of $T$.
\end{proof}

\begin{lemme}\label{lemme:connexe privé d'un ensemble ne recontrant pas sa frontiere}
Let $A$ be a connected open set of\/ $\R^d$, and $B\subseteq \R^d$ a bounded $C^1$ regular open set such that $\partial B\subseteq A$ and such that the boundary of any connected component of $B$ is connected. Then $A\setminus \overline{B}$ is connected.
\end{lemme}

\begin{proof}
As $A$ is open and connected, it is arcwise connected. Fix $x,y\in A\setminus B$ and $\Sigma$ an arc joining $x$ and $y$ in $A$. If there exists a connected component $b$ of $B$ such that $\Sigma\cap b\neq\emptyset$, any connected component $\sigma$ of $\Sigma\cap\overline{b}$ is an arc. Note that $\sigma(0)$ and $\sigma(1)$ belong to $\partial b$. But, by assumption $\partial b$ is connected, and, as $\partial b$ is a submanifold, it is then arcwise connected. Consequently, $\sigma(0)$ and $\sigma(1)$ might be joined by another arc $\tau$ in $\partial b\subseteq A$. Replacing $\sigma$ by $\tau$ in $\Sigma$, and doing the same for any connected component of $\Sigma\cap b$ and any connected component $b$ of $B$ intersecting $\Sigma$, we build an arc from $x$ to $y$ belonging to $A\setminus B$. Thus, $A\setminus B$ is arc-connected.

But, using the regularity of $B$, we can do better and show that $A\setminus\overline{B}$ is also arcwise connected. Indeed, if $B$ is $C^1$, one can define its normal vector $\vec{n}$ (and extend it to the whole $\R^d$). Then, if $\epsilon$ is  small enough, as $A$ is open, $B_\epsilon:=(I+\epsilon\vec{n})B$ satisfies the same hypothesis than $B$, but also the ancillary condition $A\setminus B_\epsilon\subseteq A\setminus\overline{B}$. Thus, one can find an arc from $x$ to $y$ through $A\setminus B_\epsilon$, hence $A\setminus\overline{B}$ is arcwise connected.
\end{proof}

\section{Symmetrisation}\label{annexe:talenti}

In this appendix, we prove some relations regarding signed symmetrisations. Lemma \ref{lemme:symmetrisation de talenti} is a useful property for proving Proposition \ref{prop:two-balls problem} which translates the equimeasurability of $f$ and $f^\dagger$ when the underlying domain is decomposed into two disjoint sets. Lemma \ref{lemme:les symétrisations signées coincident} asserts that the two natural signed rearrangments defined in section \ref{sec:talenti} - Talenti and signed Schwarz symmetrisations - actually coincide. Lemma \ref{lemme:symétrisation de Talenti croissante} is a generalisation of Hardy-Littlewood inequality used in the proof of Theorem \ref{thm:talenti modifie}.

\begin{lemme}\label{lemme:symmetrisation de talenti}
Let $p\in\N$. Let $\omega$ be a bounded open set and $f\in L^p(\omega)$. Let $\omega_+,\omega_-\subseteq\omega$ be disjoint sets such that $|\omega|=|\omega_+|+|\omega_-|$. Then
$$
\int_{\omega_+^*}\left(f^\dagger\right)^p+(-1)^p\int_{\omega_-^*}\left((-f)^\dagger\right)^p=\int_{\omega}f^p.
$$
\end{lemme}

\begin{proof}
We recall that since for any $s\in(0,|\omega|)$, $f_+^*(s)\neq 0\Leftrightarrow s<|\{f>0\}|$ and $f_-^*(|\omega|-s)\neq 0\Leftrightarrow s>|\omega|-|\{f<0\}|\geq|\{f>0\}|$, it is impossible for $f_+^*(s)$ and $f_-^*(|\omega|-s)$ to be simultaneously nonzero. Then, using Newton's binomial expansion, we get
$$
\int_{\omega_+^*}\left(f^\dagger\right)^p=\int_0^{|\omega_+|}\left(f^\dagger\right)^p=\int_0^{|\omega_+|}\left(f_+^*\right)^p(s)+(-1)^p\int_0^{|\omega_+|}\left(f_-^*\right)^p(|\omega|-s).
$$
Due to the assumption $|\omega|=|\omega_+|+|\omega_-|$, a change of variable gives
$$
\int_{\omega_+^*}\left(f^\dagger\right)^p=\int_0^{|\omega_+|}\left(f_+^*\right)^p(s)+(-1)^p\int_{|\omega_-|}^{|\omega|}\left(f_-^*\right)^p(s).
$$
Replacing $f$ with $-f$ and exchanging $\omega_+$ and $\omega_-$ in the previous relation, we find
$$
\int_{\omega_-^*}\left((-f)^\dagger\right)^p=\int_0^{|\omega_-|}\left(f_-^*\right)^p(s)+(-1)^p\int_{|\omega_+|}^{|\omega|}\left(f_+^*\right)^p(s).
$$
The combination of the last identities turns into the next one
$$
\int_{\omega_+^*}\left(f^\dagger\right)^p+(-1)^p\int_{\omega_-^*}\left((-f)^\dagger\right)^p=\int_{\omega^*}\left(f_+^*\right)^p+(-1)^p\int_{\omega^*}\left(f_-^*\right)^p.
$$
Thanks to the equimeasurability of Schwarz symmetrisation and using again Newton's expansion, we eventually get as desired
$$
\int_{\omega_+^*}\left(f^\dagger\right)^p+(-1)^p\int_{\omega_-^*}\left((-f)^\dagger\right)^p=\int_\omega\left(f_+\right)^p+(-1)^p\int_\omega\left(f_-\right)^p=\int_\omega f^p.
$$
\end{proof}

\begin{lemme}\label{lemme:les symétrisations signées coincident}
Let $\omega$ be a bounded open set and $f\in L^\infty(\omega)$. Then,
$$
f^*=f^\dagger, \quad \text{a.e. in }\omega^*.
$$
\end{lemme}

\begin{proof}
We will prove that $f^*$ and $f^\dagger$ have same moments. Then, because $f^*$ and $f^\dagger$ are bounded, this will show that they have same law, and in particular same distribution function. Since they are both radially symmetric and nonincreasing, their upper level sets will be balls of same volume. In particular, their upper level sets will coincide almost everywhere, and due to the layer-cake representation, $f^*$ and $f^\dagger$ will in turn coincide almost everywhere. To show that they share same moments, let us fix some $p\in\N$. First, by construction $f^*$ and $f$ are equimesurable, hence we have
$$
\int_\omega \left(f^*\right)^p=\int_\omega f^p.
$$
On the other hand, apply Lemma \ref{lemme:symmetrisation de talenti} with $\omega_+=\omega$ and $\omega_-=\emptyset$ to observe that
$$
\int_\omega \left(f^\dagger\right)^p=\int_\omega f^p,
$$
which concludes.
\end{proof}

\begin{lemme}\label{lemme:symétrisation de Talenti croissante}
Let $\omega$ be a bounded open set and $f$ a measurable function defined on $\omega$. Then, for all $A\subseteq\omega$,
$$
(f|_A)^*\leq f^*|_{A^*}.
$$
\end{lemme}

\begin{proof}
It is enough to prove a corresponding inequality for the distribution functions of $f|_A$ and $f^*|_{A^*}$, that is, for any $t\in\R$,
$$
|\{f>t\}\cap A|\leq|\{f^*>t\}\cap A^*|.
$$
Due to the relation $\{f^*>t\}=\{f>t\}^*$, the previous is equivalent to
$$
\int\1_A\1_{\{f>t\}}\leq \int\1_{A^*}\1_{\{f>t\}^*}=\int\1_{A}^*\1_{\{f>t\}}^*,
$$
which holds thanks to Hardy-Littlewood inequality.
\end{proof}

\section{Asymmetric two-ball problem}\label{annexe:pb 2 boules}

This section is devoted to the technical result needed to complete the proof of the central Theorem \ref{thm:pb aux 2 boules} (see inequality (\ref{eq:Fv<0})), stating that the asymmetric two-ball problem reaches its minimum when one ball vanishes.

\begin{theoreme}\label{thm:Fv<0}
Let $d\geq 4$, set $\nu=d/2-1$, and assume that
\begin{enumerate}
\item There exists a sequence $0=x_0<x_1<...<x_n<x_{n+1}=a_I$ such that (\ref{eq:Fv en xi}) holds for each $i\in]0,n]$, and (\ref{eq:Gv en xi}) holds.
\item There exists a sequence $a_S=y_0<y_1<...<y_m<y_{m+1}=1$ such that (\ref{eq:Fv en yi}) holds for each $i\in[0,m[$, and (\ref{eq:Fv'>0 sur ]ym,1[}) holds.
\end{enumerate}
Then, for every $a\in A=[0,a_I[\cup]a_S,1]$,
\begin{equation}\label{eq:Fv<0 2eme occurence}
F_\nu(k_\nu,a)\leq 0
\end{equation}
\end{theoreme}

\begin{proof}
We use the same notations as in the proof of Theorem \ref{thm:pb aux 2 boules}. First, let us remark that $K$ is a positive increasing function of $a$, and that when $a\in A_I:=[0,a_I[$, thanks to (\ref{eq:equations en I et S}), $k_\nu Ka\in[0,k_\nu K(a_I)a_I[\subseteq [0,k_\nu K(a_S)a_S[=[0,j_\nu[$. On the other hand, $b$ being positive decreasing with respect to $a$, then according to (\ref{eq:equations en I et S}) once again, $k_\nu b\in]k_\nu b_S,k_\nu]=]j_\nu,k_\nu]$. In this context, we see that, on $A_I$, $F_\nu(k_\nu,a)$ is the sum of two rival terms: the first, $f_\nu(k_\nu Ka)$, is positive increasing, and the second, $K^df_\nu(k_\nu b)$, negative decreasing (as a product of a positive increasing and a negative decreasing function) with respect to the variable $a$. Then, we can try using a \enquote{zigzag argument} (as in Lemma 2, case $n=3$ of \cite{ashbaugh-benguria}) to show that $F_\nu(k_\nu,a)\leq0$. More precisely, if one is able to exhibit a sequence of points $0=x_0<x_1<...<x_n<x_{n+1}=a_I$ such that for all $0\leq i\leq n$,
\begin{equation*}
f_\nu(k_\nu K(x_{i+1})x_{i+1})+K(x_i)^df_\nu(k_\nu b(x_i)) \leq 0,
\end{equation*}
then we can conclude that (\ref{eq:Fv<0 2eme occurence}) holds on $A_I$. This is because, for all $a\in A_I$, $a$ belongs to some interval $[x_i,x_{i+1}]$, and hence $F_\nu(k_\nu,a)\leq f_\nu(k_\nu K(x_{i+1})x_{i+1})+K(x_i)^df_\nu(k_\nu b(x_i))$ thanks to the monotonic behavior of each term. Of course, it is impossible to find such sequence since for $i=0$, $f_\nu(k_\nu K(x_{i+1})x_{i+1})>0$ whereas $K(x_i)=0$. Nevertheless, apart from the case $i=0$, one can expect to show this inequality, which exlpains why equation (\ref{eq:Fv en xi}) is an assumption of this theorem. This assumption induces that $F_\nu(k_\nu,a)\leq0$ on $[x_1,a_I[$.

Now if $a\in A_S:=]a_S,1]$, (\ref{eq:equations en I et S}) implies that $k_\nu Ka\in]j_\nu,k_\nu]$ and that $k_\nu b\in[0,k_\nu b_I[\subset[0,j_\nu[$, so that, since $K\leq1$, $F(k_\nu,a)\leq \tilde{F}_\nu(k_\nu,a):=f_\nu(k_\nu Ka)+f_\nu(k_\nu b)$. Then, $\tilde{F}_\nu(k_\nu,a)$ is the sum of the negative increasing function of $a$ $f_\nu(k_\nu Ka)$ and the positive decreasing function of $a$ $f_\nu(k_\nu b)$. Analogously as what have been done for $a\in A_I$, it is clear that (\ref{eq:Fv<0 2eme occurence}) holds on $A_S$ as soon as one is able to find $a_S=y_0<y_1<...<y_m<y_{m+1}=1$ such that $\forall 0\leq i\leq m$
\begin{equation*}
f_\nu(k_\nu K(y_{i+1})y_{i+1})+f_\nu(k_\nu b(y_i))\leq0.
\end{equation*}
Of course, as for (\ref{eq:Fv en xi}), this inequality always fails for $i=m$. But since (\ref{eq:Fv en yi}) is assumed to be true for all $i\neq m$, we get that $F_\nu(k_\nu,a)\leq0$ on $]a_S,y_m]$.

The discussion above shows that it remains to prove $F_\nu(k_\nu,a)\leq0$ only near $a=0$ and $a=1$, that is on $[0,x_1[$ and on $]y_m,1]$. For this purpose, one can compute the derivative of $a\mapsto F_\nu(k_\nu,a)$. Indeed, since (\ref{eq:Fv<0 2eme occurence}) is already known for $a=0$ and $a=1$, it will hold on the whole $[0,x_1[\cup]y_m,1]$ if one can show that
\begin{equation}\label{eq:Fv'<0}
\partial_a F_\nu(k_\nu,a)< 0,\qquad \forall a\in]0,x_1[,
\end{equation}
and that
\begin{equation}\label{eq:Fv'>0}
\partial_a F_\nu(k_\nu,a)> 0,\qquad \forall a\in]y_m,1[.
\end{equation}
Actually - as pointed out in \cite{ashbaugh-benguria} -, since (\ref{eq:Fv<0 2eme occurence}) is known asymptotically in a neighbourhood of $0$ and $1$ (this claim will be proved at the end, but see also Lemma 1 \cite{ashbaugh-benguria}), it is enough to show (\ref{eq:Fv'<0}) and (\ref{eq:Fv'>0}) only when $F_\nu(k_\nu,a)=0$. That's why, for the rest of the proof, we make the ancillary assumption that $F_\nu(k_\nu,a)=0$ unless otherwise mentionned. We have to compute $\partial_a F_\nu(k_\nu,a)$. For readability, we omit from time to time the variable $a$, and write $k$ instead of $k_\nu$. Moreover, we use, according to appendix 1 of \cite{ashbaugh-benguria} and to the identity $2\nu+1=d-1$, that
\begin{equation}\label{eq:fnu'}
\begin{split}
f_\nu'(r) & =\left[\left(\frac{J_{\nu+1}(r)}{J{\nu}(r)}-\frac{I_{\nu+1}(r)}{I{\nu}(r)}\right)\left(\frac{J_{\nu+1}(r)}{J{\nu}(r)}+\frac{I_{\nu+1}(r)}{I{\nu}(r)}\right)+2\right]r^{2\nu+1} \\
& =2r^{d-1}+\left(\frac{J_{\nu+1}(r)}{J_\nu(r)}-\frac{I_{\nu+1}(r)}{I_\nu(r)}\right)f_\nu(r).
\end{split}
\end{equation}
Then,
\begin{align*}
\partial_a F_\nu(k,a)= & 2k(aK'+K)(kKa)^{d-1} + k(aK'+K)\left(\frac{J_{\nu+1}}{J_\nu}(k Ka)-\frac{I_{\nu+1}}{I_\nu}(k Ka)\right)f_\nu(kKa) \\
& +dK'K^{d-1}f_\nu(kb)+2k K^db'(kb)^{d-1}+kK^db'\left(\frac{J_{\nu+1}}{J_\nu}(k b)-\frac{I_{\nu+1}}{I_\nu}(k b)\right)f_\nu(kb).
\end{align*}
The hypothesis $F_\nu(k,a)=0$ made above reads $f_\nu(kKa)=-K^df_\nu(kb)$. This allows writing
\begin{equation}\label{eq:Fv' expression}
\partial_a F_\nu(k,a)=2k^d[(aK'+K)(aK)^{d-1}+b'b^{d-1}K^d]+dK'K^{d-1}f_\nu(kb)+ka^{d-1}G_\nu(a) f_\nu(kKa),
\end{equation}
where
\begin{equation*}
G_\nu(a)=\frac{aK'+K}{a^{2\nu+1}}\left(\frac{J_{\nu+1}}{J_\nu}(kKa)-\frac{I_{\nu+1}}{I_\nu}(kKa)\right)-\frac{b'}{a^{2\nu+1}}\left(\frac{J_{\nu+1}}{J_\nu}(kb)-\frac{I_{\nu+1}}{I_\nu}(kb)\right).
\end{equation*}
We will study each of the three terms involved in (\ref{eq:Fv' expression}). For that we define
$$
T_1:=(aK'+K)(aK)^{d-1}+b'b^{d-1}K^d,\qquad T_2:=dK'K^{d-1}f_\nu(kb),\qquad T_3:=ka^{d-1}G_\nu(a) f_\nu(kKa).
$$
The second term, $T_2$, is the easiest to analyse. Indeed, if $a\in A_I^*:=]0,a_I[$, then $b\in]b_S,1[$, so that (according to (\ref{eq:equations en I et S})) $kb\in]j_\nu,k_\nu[$ and eventually $f_\nu(kb)<0$. On the other hand, if $a\in A_S^*:=]a_S,1[$, because $K$ is increasing and thanks to (\ref{eq:equations en I et S}), $kKa\in]j_\nu,k_\nu[$, and as before $f_\nu(kKa)<0$. The relation $f_\nu(kKa)=-K^df_\nu(kb)$ show in turn that $f_\nu(kb)>0$. In summary, because $K,K'>0$ on $]0,1[$ (see (\ref{eq:K'})), we have $T_2<0$ when $a\in A_I^*$ and $T_2>0$ when $a\in A_S^*$.

Let's tackle $T_1$. With the notation $A=a^{d-1}$ and $B=b^{d-1}$, one computes
\begin{equation}\label{eq:K'}
b'=-\frac{A}{B},\qquad K=\frac{A}{1+B},\qquad K'=\frac{(d-1)(b+1)K^2}{abA},
\end{equation}
and\begin{align*}
T_1=(aK)^d\frac{K'}{K}.
\end{align*}
We see that $K'$ and $T_1$ are products of positive increasing fonctions on $]0,1[$, so that they are themselves positive increasing. Because the properties of $K'$ and $T_1$ will be helpful to conclude the proof, we encapsulate it in the following lemma.
\begin{lemme}\label{lemme: T1}
$K'$ and $T_1$ are positive increasing on $]0,1[$, and $K',T_1\sim\frac{d-1}{b}\to\infty$ when $a\to1$.
\end{lemme}
This result gives that $T_1>0$ on $A_S$ which was desirable. The problem is that we do not have $T_1<0$ on $A_I^*$. Hopefully, as we shall see in a few moment, its positivity near $a=0$ will be compensated by the negativity of $T_2$. Indeed, recall that \mbox{$T_2=dK'K^{d-1}f_\nu(kb)=-d\frac{K'}{K}f_\nu(kKa)$}. To analyse the behavior of $f_\nu(kKa)$ around $a=0$ and then go further in the analysis of $2k^dT_1+T_2$, we use the identities given in \cite[Lemma 2]{ashbaugh-benguria}, which are
\begin{align}
\frac{J_{\nu+1}}{J_\nu}(r) & =2r\sum_{m\geq1}\frac{1}{j_{\nu,m}^2-r^2}, \label{eq:DSE du rapport des Jv}\\
\frac{I_{\nu+1}}{I_\nu}(r) & =2r\sum_{m\geq1}\frac{1}{j_{\nu,m}^2+r^2}.\label{eq:DSE du rapport des Iv}
\end{align}
Then, for all $r\in ]0,j_\nu[$,
$$
f_\nu(r)=2r^d\sum_{m\geq1}\frac{2j_{\nu,m}^2}{j_{\nu,m}^4-r^4}> 4r^d\sum_{m\geq1}\frac{1}{j_{\nu,m}^2}.
$$
But thanks to \cite{giusti-mainardi}, $\sum_{m\geq1}\frac{1}{j_{\nu,m}^2}=\frac{1}{4(\nu+1)}=1/(2d)$. Therefore, for all $a\in A_I^*$, $kKa\in ]0,j_\nu[$ and thus
$$
f_\nu(kKa)> \frac{2}{d}(kKa)^d.
$$
Hence, for all $a\in A_I^*$,
$$
T_2<-2\frac{K'}{K}(kKa)^d=-2k^dT_1.
$$
Then, as desired, $2k^dT_1+T_2< 0$ on $A_I^*$.

We now study $T_3=ka^{d-1}G_\nu(a) f_\nu(kKa)$, the last term in (\ref{eq:Fv' expression}). Notice that thanks to the the fact already mentionned that $f_\nu(kKa)$ is positive when $a\in A_I^*$ and negative when $a\in A_S^*$,  $T_3$ will be negative on $A_I^*$ and positive on $A_S^*$ if and only if $G_\nu< 0$ on $A_I^*$ and on $A_S^*$. Therefore, let us write $G_\nu=G_1+G_2$ where we recall $G_1=(aK'+K)(kK)^{d-1}g_\nu(kKa)$ and $G_2=k^{d-1}g_\nu(kb)$ with
$$
g_\nu(r)=\frac{1}{r^{d-1}}\left[\frac{J_{\nu+1}}{J_\nu}(r)-\frac{I_{\nu+1}}{I_\nu}(r)\right].$$
As in \cite[Lemma 2]{ashbaugh-benguria}, one uses, thanks to (\ref{eq:DSE du rapport des Jv}) and (\ref{eq:DSE du rapport des Iv}), that
$$
g_\nu(r)=4r^{4-d}\sum_{m\geq1}\frac{1}{j_{\nu,m}^4-r^4}=:4r^{4-d}S_\nu(r).
$$
Here, we will need a new lemma.

\begin{lemme}\label{lemme:Snu}
$S_\nu$ is positive increasing on $[0,j_\nu[$ and negative increasing on $]j_\nu,k_\nu]$.
\end{lemme}

\begin{proof}
The fact that $S_\nu$ is increasing is obvious as it is the case for each term involved in the sum. The fact that it is positive on $[0,j_\nu[$ is also easy because this property holds for each term. The fact that $S_\nu$ is negative on $]j_\nu,k]$ is not as easy. It comes from the identites \mbox{$4r^3S_\nu(r)=\frac{J_{\nu+1}}{J_\nu}(r)-\frac{I_{\nu+1}}{I_\nu}(r)$} and \mbox{$\frac{J_{\nu+1}}{J_\nu}(k)+\frac{I_{\nu+1}}{I_\nu}(k)=f_\nu(k)k^{1-d}=0$}, showing that \mbox{$S_\nu(k)< 0\Longleftrightarrow\frac{J_{\nu+1}}{J_\nu}(k)<0$}. But as near $r=0$, \mbox{$J_\nu(r)\sim\frac{1}{\Gamma(\nu+1)}(r/2)^\nu$}, we have that $J_\nu$ is positive near zero. Moreover, it is known that any positive zero of $J_\nu$ is simple, so that $J_\nu$ is negative between $j_{\nu,1}$ and $j_{\nu,2}$. This allows to conclude that $J_\nu(k)<0$ (recall $j_{\nu,1}<k<j_{\nu,2}$). On the other hand, thanks to \cite[Lemma 3]{alkharsani-baricz-pogany}, we have that $k<j_{\nu+1,1}$ and hence that $J_{\nu+1}(k)>0$ since $J_{\nu+1}$ is positive between $0$ and $j_{\nu+1,1}$. Therefore, we conclude that for all $r\in]j_\nu,k[$, $S_\nu(r)< S_\nu(k)<0 $.
\end{proof}

At this point, the proof of (\ref{eq:Fv<0 2eme occurence}) will differ substantially for $a\in A_I^*$ and for $a\in A_S^*$.

\underline{$F_\nu(k,a)\leq 0$ for $a\in ]0,x_1[$}:

Due to the fact that $S_\nu$ is negative increasing on $]j_\nu,k[$, $g_v$ is negative increasing on $]j_\nu,k[$ as the product of a positive nonincreasing function (recall that here $d\geq 4$) and a negative increasing function. Then, because $b$ is decreasing, $G_2$ is negative decreasing on $A_I^*$. On the other hand, $K'$ is positive increasing with respect to $a$ on $]0,1[$. Moreover,
$$
(kK)^{d-1}g_\nu(kKa)=4K^2\frac{(ka)^3}{1+B}S_\nu(kKa).
$$
The right hand member above is the product of a positive increasing function on $A_I^*$ with $S_\nu(kKa)$, which is positive increasing on $A_I$. Indeed, we recall that $a\in A_I\Rightarrow kKa\in[0,j_\nu[$. Eventually, $G_1$ is a product of positive increasing functions on $A_I^*$, and therefore is itself positive increasing on $A_I^*$.

To sum up, on $A_I^*$, $G_\nu$ is the sum of the positive increasing function $G_1$ and the negative decreasing function $G_2$. Therefore, $G_\nu<0$ on $]0,x_1[$ if 
\begin{equation*}
G_1(x_1)+G_2(x_0)\leq0.
\end{equation*}
As (\ref{eq:Gv en xi}) is assumed to hold, we get that $\partial_a F_\nu(k,a)<0$ on $]0,x_1[$ whenever $F_\nu(k,a)=0$. Now, to conclude that (\ref{eq:Fv<0 2eme occurence}) holds on $]0,x_1[$, it remains only to show that (\ref{eq:Fv<0 2eme occurence}) held in some neighbourhood of $0$. To do so, we perform an asymptotic analysis. Observe that, thanks to (\ref{eq:DSE du rapport des Jv}), (\ref{eq:DSE du rapport des Iv}) and \cite{giusti-mainardi},
\begin{equation}\label{eq:asymptote fnu en 0}
f_\nu(r)\underset{r\to0}{\sim}\frac{2}{d}r^d
\end{equation}
On the other hand, using (\ref{eq:fnu'}), we find that $\frac{d}{d\omega}f_\nu(\omega^{\frac{1}{d}})\xrightarrow[\omega\to k^d]{}\frac{2}{d}$, from which we get
\begin{equation}\label{eq:asymptote fnu en k}
f_\nu(r)\underset{r\to k}{\sim}\frac{2}{d}(r^d-k^d).
\end{equation}
As a consequence, $f_\nu(kKa)\underset{a\to0}{\sim} -K^df_\nu(kb)$ and one needs to exhibit the next term in the asymptotic expansions of $f_\nu(kKa)$ and $K^df_\nu(kb)$ in order to conclude. For that, plug (\ref{eq:asymptote fnu en 0}) into (\ref{eq:fnu'}), and use the definition of $S_\nu$ to find
$$
f_\nu'(r)\underset{r\to 0}{=}2r^{d-1}+\frac{8}{d}S_\nu(0)r^{d+3}+o(r^{d+3})
$$
From this we obtain after integration
$$
f_\nu(r)\underset{r\to 0}{=}\frac{2}{d}r^d+\frac{8}{d}S_\nu(0)\frac{r^{d+4}}{d+4}+o(r^{d+4})
$$
Similarly, plug (\ref{eq:asymptote fnu en k}) into (\ref{eq:fnu'}) and find
$$
\frac{d}{d\omega}f_\nu(\omega^{\frac{1}{d}})\underset{\omega\to k^d}{=}\frac{2}{d}+\frac{2}{d}(\omega-k^d)4k^3S_\nu(k)k^{\frac{1-d}{d}}+o(\omega-k^d)
$$
After integration, we obtain
$$
f_\nu(r)\underset{r\to k}{=}\frac{2}{d}(r^d-k^d)+\frac{4}{d}(r^d-k^d)^2k^3S_\nu(k)k^{\frac{1-d}{d}}+o((r^d-k^d)^2)
$$
Eventually, since $(Ka)^{d+4}\underset{a\to 0}{=}o(K^da^{2d})$, we get
\begin{align*}
f_\nu(kKa)+K^df_\nu(kb) & \underset{a\to 0}{=}\frac{8}{d(d+4)}S_\nu(0)(kKa)^{d+4}+\frac{4}{d}S_\nu(k)k^{\frac{2d+1}{d}}(ka)^{2d}K^d+o((Ka)^{d+4}+K^da^{2d})\\
& \underset{a\to 0}{=}\frac{4}{d}S_\nu(k)k^{\frac{2d+1}{d}}(ka)^{2d}K^d+o(K^da^{2d})
\end{align*}
Since $S_\nu(k)<0$ due to Lemma \ref{lemme:Snu}, we see that $F_\nu(k,a)$ is negative when $a$ lies in some neighbourhood of $0$.

\newpage
\underline{$F_\nu(k,a)\leq 0$ for $a\in ]y_m,1[$}:

In the case $a\in A_S^*$, we are not able to show that $G_\nu< 0$, and indeed at least when $d\geq6$, $G_\nu(a)\to\infty$ when $a\to1$. Nevertheless, as we shall see in the next lines, we do have $G_\nu(a)f_\nu(kKa)=-K^dG_\nu(a)f_\nu(kb)\underset{a\to1}{=}O(b^4+b^{d})\to0$. But because $T_1\to\infty$ when $a\to1$, it will dominate $T_3$ in (\ref{eq:Fv' expression}), enforcing $\partial_a F_\nu(k_\nu,a)>0$ near $a=1$. However one has to quantify the predominance of $T_3$ with respect to $T_1$. For this purpose, one estimates, for $a\in ]y_m,1[$, $|G_1(a)|\leq (|K'|+1)k^{d-1}|g_\nu(kKa)|$. Now $g_\nu$ being, as $S_\nu$, negative increasing on $]j_\nu,k[$, we obtain (using also (\ref{eq:K'})) that for $a\in ]y_m,1[$,
\begin{equation}\label{eq:|G1|}
|G_1(a)|< \left[\frac{d-1}{y_m^d}\left(1+\frac{1}{b}\right)+1\right] k^{d-1}|g_\nu(kK(y_m)y_m)|.
\end{equation}
For $G_2$, we use again that $g_\nu(r)=4r^{4-d}S_\nu$, and the fact that $S_\nu$ is positive increasing on $[0,j_\nu[$, yielding, for all $a\in]y_m,1[$,
\begin{equation}\label{eq:|G2|}
|G_2(a)|< 4k^{d-1}(kb)^{4-d}S_\nu(kb(y_m))=k^{d-1}b^{4-d}\frac{g_\nu(kb(y_m))}{b(y_m)^{4-d}}.
\end{equation}
We have also to control $f_\nu(kKa)=-K^df_\nu(kb)$. To do so, we use again (\ref{eq:DSE du rapport des Jv}) and (\ref{eq:DSE du rapport des Iv}) to conclude that $r\mapsto r^{-d}f_\nu(r)$ is positive increasing on $[0,j_\nu[$, and hence that, for $a\in]y_m,1[$,
\begin{equation}\label{eq:|fv|}
|f_\nu(kb)|< (kb)^d(kb(y_m))^{-d}f_\nu(kb(y_m))=b^db(y_m)^{-d}f_\nu(kb(y_m))
\end{equation}
Combining (\ref{eq:|G1|}), (\ref{eq:|G2|}) and (\ref{eq:|fv|}), we have on $]y_m,1[$,
\begin{align*}
|T_3|\leq k|G_\nu||f_\nu(kb)|< & k^d|g_\nu(kK(y_m)y_m)|f_\nu(kb(y_m))b(y_m)^{-d}b^d\left[\frac{d-1}{y_m^d}\left(1+\frac{1}{b}\right)+1\right] \\
& +k^d(g_\nu f_\nu)(kb(y_m))b(y_m)^{-4}b^4 \\
\underset{a\to1}{=} & O(b^4+b^{d-1}).
\end{align*}
The point is that the right-hand member in the previous inequality is an increasing function of the variable $b< b(y_m)$, so that we get for all $a\in]y_m,1[$,
\begin{align*}
|T_3|< & k^d|g_\nu(kK(y_m)y_m)|f_\nu(kb(y_m))\left[\frac{d-1}{y_m^d}\left(1+\frac{1}{b(y_m)}\right)+1\right] +k^d(g_\nu f_\nu)(kb(y_m)).
\end{align*}
On the other, because $T_1$ is positive increasing on $A_S$, (\ref{eq:Fv'>0}) will be ensured as long as one is able to tune $y_m$ such that
\begin{align*}
0\leq & 2T_1(y_m) -f_\nu(kb(y_m))|g_\nu(kK(y_m)y_m)|\left[\frac{d-1}{y_m^d}\left(1+\frac{1}{b(y_m)}\right)+1\right] -(f_\nu g_\nu)(kb(y_m)).
\end{align*}
Because (\ref{eq:Fv'>0 sur ]ym,1[}) is an assumption, we conclude that $\partial_a F_\nu(k,a)>0$ on $]y_m,1[$ whenever $F_\nu(k,a)=0$. Hence, to show that (\ref{eq:Fv<0 2eme occurence}) holds on $]y_m,1[$, it remains to show that it does asymptotically near $a=1$. Recalling (\ref{eq:asymptote fnu en 0}), one finds
$$
f_\nu(kb)\underset{a\to1}{\sim}\frac{2}{d}(kb)^d.
$$
On the other hand, (\ref{eq:asymptote fnu en k}) gives
$$
f_\nu(kKa)\underset{a\to1}{\sim}\frac{2}{d}k^d((Ka)^d-1).
$$
Note that, since $K\xrightarrow[a\to1]{}1$,
$$
(Ka)^d-1=(Ka-1)\sum_{n=0}^{d-1}(Ka)^n\underset{a\to1}{\sim}-d\frac{b^d+B}{1+B}\underset{a\to1}{\sim}-db^{d-1}.
$$
From this we deduce that $K^df_\nu(kb)\underset{a\to1}{=}o(f_\nu(kKa))$ and that $F_\nu(k,a)\underset{a\to1}{\sim}-2k^db^{d-1}$. As a consequence, $F_\nu(k,a)<0$ in some neighbourhood of $a=1$, which concludes.
\end{proof} 

\section{Numerical computations}\label{sec:numerique}

In this section, we would like to check the assumptions (\ref{eq:condition necessaire kv jv}), (\ref{eq:Fv en xi}), (\ref{eq:Gv en xi}), (\ref{eq:Fv en yi}), and (\ref{eq:Fv'>0 sur ]ym,1[}) of Theorem \ref{thm:pb aux 2 boules} for several dimensions. First, it is needed to compute the quantitities $j_v$, $k_v$, $a_I$ and $a_S$. Hopefully, the $m$-th first positive zero $j_{\nu,m}$ of the Bessel function $J_\nu$ can be approximated by computing the eigenvalues of some matrix as explained in \cite{ikebe-kikuchi-fujishiro} Theorem 2.1 and 2.2. On the other hand, $k_\nu$ is the only zero of the increasing function $f_\nu$ between $j_{\nu,1}$ and $j_{\nu,2}$ (recall that $j_{\nu,i},i\in\N^*$ are the positive zeros of $J_\nu$), hence it can be obtained by dichotomy. Similarly, $a_S$ being the only zero of the increasing function $a\in[0,1]\mapsto aK(a)-\frac{j_\nu}{k_\nu}$, it shall also be computed by dichotomy. Eventually, $a_I$ is given by the formula $a_I^d=1-\frac{j_\nu^d}{k_\nu^d}$.

Then, we need to find, on the one hand, the sequence of points $(x_1,...,x_n)$ fulfilling hypotheses (\ref{eq:Fv en xi}) and (\ref{eq:Gv en xi}); and, on the other hand, the sequence $(y_1,...,y_m)$ fulfilling hypotheses (\ref{eq:Fv en yi}), and (\ref{eq:Fv'>0 sur ]ym,1[}). It seems not trivial to give an efficient algorithm exhibiting automatically these points, and we rather find the two sequences \enquote{manually}. That's why we restricted the study to the dimensions $d=4,5,6,7,8,9$. We precise that the numerical computations have been performed in Python 3, using the package \verb|special| from the module \verb|scipy|, which contains the functions \verb|jv|, and \verb|iv|, corresponding to the bessel functions $J_\nu$ and $I_\nu$.

Let us mention that while checking (\ref{eq:Fv en xi}), one has to use approximations of $a_I$ and $k_\nu$ by the right, namely $a_I^+$ and $k_\nu^+$, rather than by the left. This is because (see the proof of Theorem \ref{thm:Fv<0} for details) $(k,a)\mapsto f_\nu(k K a)$ is increasing in each variable on $[0,k_\nu^+]\times [0,a_I^+]$ and $(k,a)\mapsto Kf_\nu(k b)$ is increasing in $k$ and decreasing in $a$ on $[0,k_\nu^+]\times [0,a_I^+]$. Therefore, instead of (\ref{eq:Fv en xi}), one will try to verify
\begin{equation}\label{eq:Fvi}
\tag{\ref{eq:Fv en xi}'}
F_{\nu,i}:=f_\nu(k_\nu^+ K(x_{i+1}) x_{i+1})+K(x_i)^df_\nu(k_\nu^+b(x_i))\leq 0.
\end{equation}
Analogously, when checking (\ref{eq:Fv en yi}), one has to use an approximation of $a_S$ by the left, that we denote $a_S^-$; but still an approximation of $k_\nu$ by the right. That's why (\ref{eq:Fv en yi}) turns into
\begin{equation}\label{eq:Fvi'}
\tag{\ref{eq:Fv en yi}'}
F_{\nu,i}':=f_\nu(k_\nu^+K(y_{i+1})y_{i+1})+f_\nu(k_\nu^+b(y_i))\leq 0.
\end{equation}
Similar issues have to be taken into account when validating (\ref{eq:Gv en xi}) and (\ref{eq:Fv'>0 sur ]ym,1[}). More precisely, in (\ref{eq:Gv en xi}), $G_1$ should be replaced by some $\tilde{G}_1$ defined as $\tilde{G}_1(a)=(aK'+K)(k_\nu^+K)^{d-1}g_\nu(k_\nu^+Ka)$; $G_2$ should be replaced by $\tilde{G}_2$ defined as $\tilde{G}_2(a)=(k_\nu^-)^{d-1}g_\nu(k_\nu^+b)$; and thus (\ref{eq:Gv en xi}) should be replaced by
\begin{equation}\label{eq:Gvi}
\tag{\ref{eq:Gv en xi}'}
G_{\nu,0}:=\tilde{G}_1(x_1)+\tilde{G}_2(x_0)\leq0.
\end{equation}
On the other hand (\ref{eq:Fv'>0 sur ]ym,1[}) should be replaced by the condition
\begin{align}\label{eq:Gvi'}
\tag{\ref{eq:Fv'>0 sur ]ym,1[}'}
0\leq G_{\nu,m}':= 2T_1(y_m) -f_\nu(k_\nu^+b(y_m))|g_\nu(k_\nu^-K(y_m)y_m)|\left[\frac{d-1}{y_m^d}\left(1+\frac{1}{b(y_m)}\right)+1\right]-(f_\nu g_\nu)(k_\nu^+b(y_m)).
\end{align}
Eventually, since $a_S$ depends on $j_\nu/k_\nu$, for computing $a_S^-$ one has to use $j_\nu^-$ and $k_\nu^+$. The same conclusion holds for computing $a_I^+$.

In Table \ref{tab:condition nécessaire}, is given the value $2\frac{j_\nu^d}{k_\nu^d}+\frac{j_\nu}{k_\nu}$ for several dimensions, which shows that (\ref{eq:condition necessaire kv jv}) is fulfilled.

\begin{table}[h!]
\centering
\begin{equation*}
\begin{array}{c||c}
d & 2j_\nu^d/k_\nu^d+j_\nu/k_\nu\\
\hline
4 & 1.7848 \\
5 & 1.7563 \\
6 & 1.7345 \\
7 & 1.7172 \\
8 & 1.7029 \\
9 & 1.6910
\end{array}
\end{equation*}
\caption{Values of $2\frac{j_\nu^d}{k_\nu^d}+\frac{j_\nu}{k_\nu}$ in terms of dimension $d$.}
\label{tab:condition nécessaire}
\end{table}

Table \ref{tab:jv et kv} contains the values of $j_\nu^-$, $k_\nu^-$, $k_\nu^+$, $a_I^+$ and $a_S^-$.

\begin{table}[h!]
\centering
\begin{equation*}
\begin{array}{c||c|c|c|c|c}
d & j_\nu^- & k_\nu^- & k_\nu^+ & a_I^+ & a_S^- \\
\hline
4 & 3.831 & 4.610 & 4.611 & 0.8507 & 0.9836 \\
5 & 4.493 & 5.267 & 5.268 & 0.8869 & 0.9879 \\
6 & 5.135 & 5.905 & 5.906 & 0.9101 & 0.9907 \\
7 & 5.763 & 6.529 & 6.530 & 0.9259 & 0.9927 \\
8 & 6.380 & 7.143 & 7.144 & 0.9373 & 0.9940 \\
9 & 6.987 & 7.748 & 7.749 & 0.9459 & 0.9951
\end{array}
\end{equation*}
\caption{Values of $j_\nu^-$, $k_\nu^-$, $k_\nu^+$, $a_I^+$, and $a_S^-$ in terms of dimension $d$.}
\label{tab:jv et kv}
\end{table}

In Table \ref{tab:Gv0, Fvi, Fvi' et Gvm'} are computed for several dimensions the approximate quantities $G_{\nu,0}$, $F_{\nu,i}$, $F_{\nu,i}'$ and $G_{\nu,m}'$ appearing in (\ref{eq:Fvi}), (\ref{eq:Gvi}), (\ref{eq:Fvi'}) and (\ref{eq:Gvi'}). We observe that appart from the last column, the value reported in each cell is negative, meaning that inequalities (\ref{eq:Fvi}), (\ref{eq:Gvi}), (\ref{eq:Fvi'}) are satisfied. Moreover, the fact that the values in the last column are positive show that inequality (\ref{eq:Gvi'}) is also true. This justifies that assumptions \ref{it:hypothese Fv<0 sur AI} and \ref{it:hypothese Fv<0 sur AS} of Theorem \ref{thm:pb aux 2 boules} are fulfilled. 

\begin{table}[h!]
\centering
\begin{equation*}
\begin{array}{c||c|c|c|c|c|c|c|c|c|c|c|c|c}
d & G_{\nu,0} & F_{\nu,1} & F_{\nu,0}' & G_{\nu,1}' \\
\hline
4 & -1.232 & -6.234 & -2.682 & 26.92 \\

5 & -1.040 & -76.82 & -52.58 & 27.50 \\

6 & -0.9603 & -123.8 & -840.3 & 28.84 \\

7 & -0.8131 & -1135 & -1336\cdot 10^1 & 30.36 \\

8 & -0.7514 & -3126 & -2165\cdot 10^2 & 31.83 \\

9 & -0.6430 & -2362\cdot 10^1 & -3803\cdot 10^3 & 33.12
\end{array}
\end{equation*}
\caption{Values of $G_{\nu,0}$, $F_{\nu,i}$, $F_{\nu,i}'$ and $G_{\nu,m}'$ in terms of dimension $d$. For each dimension, it is chosen $n=1$, $m=1$, and $y_1=0.999$. The point $x_1$ chosen depends on the dimension: for $d=4$, $x_1=0.83$, for $d=5$, $x_1=0.88$, for $d=6$, $x_1=0.90$, for $d=7$, $x_1=0.92$, for $d=8$, $x_1=0.93$, and for $d=9$, $x_1=0.94$.}
\label{tab:Gv0, Fvi, Fvi' et Gvm'}
\end{table}


\newpage
\renewcommand{\abstractname}{Aknowledgements}
\begin{abstract}
\noindent Let me thank Enea Parini and François Hamel for their valuable advice and remarks, in particular regarding Talenti's comparison principle.
\end{abstract}


\bibliographystyle{alpha} 
\bibliography{../biblio} 

\begin{thebibliography}{DRGM21}

\bibitem[AB95]{ashbaugh-benguria}
Mark~S. Ashbaugh and Rafael~D. Benguria.
\newblock On {R}ayleigh's conjecture for the clamped plate and its
  generalization to three dimensions.
\newblock {\em Duke Math. J.}, 78(1):1--17, 1995.

\bibitem[ABL97]{ashbaugh-benguria-laugesen}
Mark~S. Ashbaugh, Rafael~D. Benguria, and Richard~S. Laugesen.
\newblock Inequalities for the first eigenvalues of the clamped plate and
  buckling problems.
\newblock In {\em General inequalities, 7 ({O}berwolfach, 1995)}, volume 123 of
  {\em Internat. Ser. Numer. Math.}, pages 95--110. Birkh\"{a}user, Basel,
  1997.

\bibitem[AKBP16]{alkharsani-baricz-pogany}
Huda~A. Al-Kharsani, \'{A}rp\'{a}d Baricz, and Tibor~K. Pog\'{a}ny.
\newblock Starlikeness of a cross-product of {B}essel functions.
\newblock {\em J. Math. Inequal.}, 10(3):819--827, 2016.

\bibitem[AL96]{ashbaugh-laugesen}
Mark~S. Ashbaugh and Richard~S. Laugesen.
\newblock Fundamental tones and buckling loads of clamped plates.
\newblock {\em Ann. Scuola Norm. Sup. Pisa Cl. Sci. (4)}, 23(2):383--402, 1996.

\bibitem[BB05]{bucur-buttazzo}
Dorin Bucur and Giuseppe Buttazzo.
\newblock {\em Variational methods in shape optimization problems}, volume~65
  of {\em Progress in Nonlinear Differential Equations and their Applications}.
\newblock Birkh\"{a}user Boston, Inc., Boston, MA, 2005.

\bibitem[BL13]{buoso-lamberti13}
Davide Buoso and Pier~Domenico Lamberti.
\newblock Eigenvalues of polyharmonic operators on variable domains.
\newblock {\em ESAIM Control Optim. Calc. Var.}, 19(4):1225--1235, 2013.

\bibitem[Buc05]{bucur}
Dorin Bucur.
\newblock How to prove existence in shape optimization.
\newblock {\em Control Cybernet.}, 34(1):103--116, 2005.

\bibitem[CDS79]{coffman-duffin-shaffer}
Charles~V. Coffman, Richard~J. Duffin, and Douglas~H. Shaffer.
\newblock The fundamental mode of vibration of a clamped annular plate is not
  of one sign.
\newblock In {\em Constructive approaches to mathematical models ({P}roc.
  {C}onf. in honor of {R}. {J}. {D}uffin, {P}ittsburgh, {P}a., 1978)}, pages
  267--277. Academic Press, New York-London-Toronto, Ont., 1979.

\bibitem[Cia00]{cianchi}
Andrea Cianchi.
\newblock Second-order derivatives and rearrangements.
\newblock {\em Duke Math. J.}, 105(3):355--385, 2000.

\bibitem[Cof82]{coffman}
Charles~V. Coffman.
\newblock On the structure of solutions {$\Delta ^{2}u=\lambda u$} which
  satisfy the clamped plate conditions on a right angle.
\newblock {\em SIAM J. Math. Anal.}, 13(5):746--757, 1982.

\bibitem[DRGM21]{regibus-grossi-mukherjee}
Fabio De~Regibus, Massimo Grossi, and Debangana Mukherjee.
\newblock Uniqueness of the critical point for semi-stable solutions in {$\Bbb
  R^2$}.
\newblock {\em Calc. Var. Partial Differential Equations}, 60(1):Paper No. 25,
  13, 2021.

\bibitem[DS52]{duffin-shaffer}
Richard~J. Duffin and Douglas~H. Shaffer.
\newblock On the modes of vibration of a ring-shaped plate.
\newblock {\em Bull. Amer. Math. Soc.}, 58(6):652, 1952.

\bibitem[ES22]{eichmann-schatzle}
Sascha Eichmann and Reiner~M. Sch\"{a}tzle.
\newblock Positivity for the clamped plate equation under high tension.
\newblock {\em Ann. Mat. Pura Appl. (4)}, 201(4):2001--2020, 2022.

\bibitem[GGS10]{gazzola-grunau-sweers}
Filippo Gazzola, Hans-Christoph Grunau, and Guido Sweers.
\newblock {\em Polyharmonic boundary value problems}, volume 1991 of {\em
  Lecture Notes in Mathematics}.
\newblock Springer-Verlag, Berlin, 2010.
\newblock Positivity preserving and nonlinear higher order elliptic equations
  in bounded domains.

\bibitem[GM03]{garcia-melian}
Jorge Garc\'{\i}a-Meli\'{a}n.
\newblock On the behaviour of the first eigenfunction of the {$p$}-{L}aplacian
  near its critical points.
\newblock {\em Bull. London Math. Soc.}, 35(3):391--400, 2003.

\bibitem[GM16]{giusti-mainardi}
Andrea Giusti and Francesco Mainardi.
\newblock On infinite series concerning zeros of bessel functions of the first
  kind.
\newblock {\em Eur. Phys. J. Plus}, 131(6):206, Jun 2016.

\bibitem[Gro21]{grossi}
Massimo Grossi.
\newblock On the number of critical points of solutions of semilinear elliptic
  equations.
\newblock {\em Electron. Res. Arch.}, 29(6):4215--4228, 2021.

\bibitem[HP05]{henrot-pierre}
Antoine Henrot and Michel Pierre.
\newblock {\em Variation et optimisation de formes}, volume~48 of {\em
  Math\'{e}matiques \& Applications (Berlin)}.
\newblock Springer, Berlin, 2005.
\newblock Une analyse g\'{e}om\'{e}trique.

\bibitem[IKF91]{ikebe-kikuchi-fujishiro}
Yasuhiko Ikebe, Yasushi Kikuchi, and Issei Fujishiro.
\newblock Computing zeros and orders of {B}essel functions.
\newblock In {\em Proceedings of the {I}nternational {S}ymposium on
  {C}omputational {M}athematics ({M}atsuyama, 1990)}, volume~38, pages
  169--184, 1991.

\bibitem[Kaw85]{kawohl}
Bernhard Kawohl.
\newblock {\em Rearrangements and convexity of level sets in {PDE}}, volume
  1150 of {\em Lecture Notes in Mathematics}.
\newblock Springer-Verlag, Berlin, 1985.

\bibitem[Kes06]{kesavan}
Srinivasan Kesavan.
\newblock {\em Symmetrization \& applications}, volume~3 of {\em Series in
  Analysis}.
\newblock World Scientific Publishing Co. Pte. Ltd., Hackensack, NJ, 2006.

\bibitem[Kri20]{kristaly20}
Alexandru Krist\'{a}ly.
\newblock Fundamental tones of clamped plates in nonpositively curved spaces.
\newblock {\em Adv. Math.}, 367:107113, 39, 2020.

\bibitem[Kri22]{kristaly22}
Alexandru Krist\'{a}ly.
\newblock Lord {R}ayleigh's conjecture for vibrating clamped plates in
  positively curved spaces.
\newblock {\em Geom. Funct. Anal.}, 32(4):881--937, 2022.

\bibitem[Ley23]{leylekian}
Roméo Leylekian.
\newblock Sufficient conditions yielding the rayleigh conjecture for the
  clamped plate, 2023.
\newblock arXiv:2302.06313.

\bibitem[Lim88]{lima}
Elon~L. Lima.
\newblock The {J}ordan-{B}rouwer separation theorem for smooth hypersurfaces.
\newblock {\em Amer. Math. Monthly}, 95(1):39--42, 1988.

\bibitem[Moh75]{mohr}
Ernst Mohr.
\newblock \"{U}ber die {R}ayleighsche {V}ermutung: unter allen {P}latten von
  gegebener {F}l\"{a}che und konstanter {D}ichte und {E}lastizit\"{a}t hat die
  kreisf\"{o}rmige den tiefsten {G}rundton.
\newblock {\em Ann. Mat. Pura Appl. (4)}, 104:85--122, 1975.

\bibitem[Nad95]{nadirashvili}
Nikolai~S. Nadirashvili.
\newblock Rayleigh's conjecture on the principal frequency of the clamped
  plate.
\newblock {\em Arch. Rational Mech. Anal.}, 129(1):1--10, 1995.

\bibitem[Ped58]{pederson}
R.~N. Pederson.
\newblock On the unique continuation theorem for certain second and fourth
  order elliptic equations.
\newblock {\em Comm. Pure Appl. Math.}, 11:67--80, 1958.

\bibitem[Pro60]{protter}
M.~H. Protter.
\newblock Unique continuation for elliptic equations.
\newblock {\em Trans. Amer. Math. Soc.}, 95:81--91, 1960.

\bibitem[Sch15]{schmidt}
Thomas Schmidt.
\newblock Strict interior approximation of sets of finite perimeter and
  functions of bounded variation.
\newblock {\em Proc. Amer. Math. Soc.}, 143(5):2069--2084, 2015.

\bibitem[Sto21]{stollenwerk}
Kathrin Stollenwerk.
\newblock Existence of an optimal domain for minimizing the fundamental tone of
  a clamped plate of prescribed volume in arbitrary dimension, 2021.
\newblock arXiv:2109.01455.

\bibitem[Sze50]{szego}
G\'{a}bor Szeg\"{o}.
\newblock On membranes and plates.
\newblock {\em Proc. Nat. Acad. Sci. U.S.A.}, 36:210--216, 1950.

\bibitem[Tal76]{talenti76}
Giorgio Talenti.
\newblock Elliptic equations and rearrangements.
\newblock {\em Ann. Scuola Norm. Sup. Pisa Cl. Sci. (4)}, 3(4):697--718, 1976.

\bibitem[Tal81]{talenti81}
Giorgio Talenti.
\newblock On the first eigenvalue of the clamped plate.
\newblock {\em Ann. Mat. Pura Appl. (4)}, 129:265--280 (1982), 1981.

\end{thebibliography}




\end{document}